\documentclass[12pt]{amsart}
\input{cyracc.def}
\font\cyr=wncyr7

\usepackage[all]{xy} 
\usepackage[usenames,dvipsnames]{pstricks}
\usepackage{amssymb,mathrsfs,euscript,enumitem,amsmath,bbold}
\usepackage{amssymb}
\usepackage[hypertex]{hyperref}
\usepackage{quoting}

\catcode`\@=11
\def\@evenfoot{\rule{0pt}{20pt}[May 1, 2020] \hfill [{\tt \jobname.tex}]}
\def\@oddfoot{\rule{0pt}{20pt}{[\tt \jobname.tex}]\hfill [May 1, 2020]}
\catcode`\@=13
\textheight9in  \def\DATE{\today}
\allowdisplaybreaks

\newtheorem{theorem}{Theorem}
\newtheorem{corollary}[theorem]{Corollary}

\newtheorem{proposition}[theorem]{Proposition}

\newtheorem*{theoremA}{Theorem~A}
\newtheorem*{theoremB}{Theorem~B}

\newtheorem*{fact}{Fact}

\theoremstyle{definition}
\newtheorem{example}[theorem]{Example}
\newtheorem{question}[theorem]{Question}
\newtheorem{remark}[theorem]{Remark}
\newtheorem{definition}[theorem]{Definition}

\def\uoC{{\underline {\EuScript C}}}
\def\uE{{\underline E}}
\def\ucan{{\underline {\can}}}
\def\uX{{\underline X}}

\def\shufflerm{\hbox{{\cyr \cyracc Sh}}}
\def\stan{{\rm st}}
\def\ush{{\rm Sh}}
\def\oucC{{\overline \ucC}}
\def\oDelta{{\overline \Delta}}
\def\ocC{{\overline \cC}}
\def\ucobar{{\underline \cobar}}
\def\ucC{{\underline \cC}}
\def\uD{{\underline \D}}
\def\In{{\rm In}}
\def\vert{{\rm Vert}}
\def\Tr{{\rm Tr}}
\def\Ker{{\rm Ker}}
\def\J{{\mathfrak I}}
\def\cC{{\EuScript C}}
\def\Alg{\hbox{$\ttP$-${\tt Oper}_1$}} 
\def\Collect{\hbox{$\ttP$-${\tt Coll}_1$}} 
\def\Collectord{\hbox{$\Ord$-${\tt Coll}_1$}} 
\def\sgn{{\rm signum}}

\DeclareMathOperator\can{{\it can}}
\DeclareMathOperator\card{{card}}
\def\cobar{{\mathbb \Omega}}
\def\eul#1{{\chi\big(#1\big)}} 
\def\ss{{\mathfrak s}}
\def\D{{\mathbb D}}
\def\-{\!-\!}
\def\+{\!+\!}
\def\uoO{{\underline\oO}}
\def\uoP{{\underline\oP}}
\def\susp{\uparrow \hskip -.35em}
\def\antishriek{{\hbox{\scriptsize\rm \raisebox{.2em}{!`}}}}
\def\desusp{\downarrow\!}
\def\pcirc{\hbox{\tiny $\diamondsuit$}}
\def\sspcirc{\hbox{\tiny $\diamondsuit^{\hbox{$\ss$}}_{\hbox{$r$}}$}}
\def\M{{\EuScript M}}
\def\K{{\underline {\EuScript K}}}
\def\f{{\gamma}}
\def\bbN{{\mathbb N}}
\def\fib{\triangleright}
\def\uAss{\underline{\hbox{$\mathscr A \hskip -.2em ss$}}}
\def\ainf{{\uAss\,}_\infty}
\def\pa{{\partial}}
\def\min{{\mathfrak M}}
\def\oC{{\EuScript C}}
\def\uoS{{\underline \oS}}
\def\oO{{\EuScript O}}
\def\Op#1{\hbox{$#1$-{\tt Oper}}}
\def\uF{\underline{\mathbb F}}
\def\coll#1#2{{\{#1(\underline #2)\}_{#2 \geq 1}}}
\def\F{{\mathbb F}}
\def\oS{{\EuScript S}}
\def\des{{\rm des}}

\def\Pterm{{\sf 1}_{\ttP}}

\def\Ord{{\mathbb \Delta_{\rm semi}}}
\def\comp{{\rm comp}}
\def\doubless#1#2{{
\def\arraystretch{.5}
\begin{array}{c}
\mbox{$\scriptstyle #1$}
\\
\mbox{$\scriptstyle #2$}
\end{array}\def\arraystretch{1}
}}
\def\Lin{\Vect}
\def\End{{\EuScript E}nd}
\def\oP{{\EuScript P}}
\def\ttO{{\tt O}}
\def\rada#1#2{{#1,\ldots,#2}}
\def\Rada#1#2#3{#1_{#2},\dots,#1_{#3}}
\def\CC{\hbox{\it CC}}
\def\Span{{\rm Span}}
\def\Set{{\tt Set}}
\def\term{\hbox {$\it per \hskip -.1em   {\mathcal A}s$}}
\def\tildeterm{\hbox {$\it per \hskip -.1em  {\widetilde {\mathcal A}s}$}}
\def\tw{\hbox {$\it per \hskip -.1em  {\widetilde {\mathcal A}s^!}$}}
\def\ssterm{\hbox {\scriptsize $\it per \hskip -.1em   {\mathcal A}s$}}
\def\PP{{\mathbb P}}
\def\id{1\!\!1}
\def\Coll{{\tt Coll}}
\def\epi{{\hbox{ $\twoheadrightarrow$ }}}
\def\bfk{{\mathbb k}}
\def\bbk{{\mathbb k}}
\def\Vect{{\tt Vec}}
\def\ot{\otimes}
\def\inv#1{{#1}^{-1}}
\def\rada#1#2{{#1,\ldots,#2}}
\def\Surj{{\tt Surj}}
\def\Ass{\underline{{\mathcal A}{\it ss}}}
\def\Fin{{\tt Fin}}
\def\RTr{{\tt RTr}}
\def\ttP{{\tt Per}}
\def\qb{quasi\-bijection}

\title[Permutads and hidden associahedron] {Permutads
via operadic categories,\\ and the hidden associahedron}

\author[M.\ Markl]{Martin Markl}

\catcode`\@=11
\address{The Czech Academy of Sciences, Institute of Mathematics, {\v Z}itn{\'a} 25,
         115 67 Prague 1, The Czech Republic}
\catcode`\@=13

\thanks{Supported by Praemium Academiae, 
grant  GA \v CR 18-07776S and RVO: 67985840.}

\begin{document}

\parskip3pt plus 1pt minus .5pt
\baselineskip 17.25pt  plus 1.5pt minus .5pt

\begin{abstract}
The present article exploits the fact that permutads (aka shuffle algebras)
are algebras over a terminal operad in a certain operadic category $\ttP$. In
the first, classical part we formulate and prove a claim
envisaged by Loday and Ronco that the cellular chains of the
permutohedra form the minimal model of the terminal permutad which is
moreover, in the sense we define, self-dual and Koszul. 
In the second part we
study Koszulity of $\ttP$-operads. Among other things we prove that the
terminal $\ttP$-operad is Koszul
self-dual. We then describe strongly homotopy permutads as algebras of its
minimal model. Our paper shall advertise
analogous future results valid in general operadic categories,
and the prominent r\^ole of operadic (op)fibrations in the related theory.
\end{abstract}

\maketitle
\bibliographystyle{plain}

\tableofcontents

\section*{Motivations and the main results}

Recall that differential graded (abbreviated dg) associative algebras
are triples \hbox{$A = (A,\bullet,d)$} consisting of a 
graded vector space $A$, an associative degree-$0$
product \hbox{$\bullet : A \ot A \to A$}, and 
a~differential $d$ which is a derivation with respect to the product. 
Associative algebras are however too rigid for some applications
in homological algebra and deformation theory. 
For instance, having an algebra $A
= (A,\bullet,d)$ as
above and a dg vector space $(A',d')$ chain homotopy equivalent to
$(A,d)$, there is in general no associative
multiplication on $(A',d')$ induced by the one on $(A,d)$.
 
The way around is to embed dg associative algebras into the category 
of {\em strongly homotopy\/} associative algebras. These objects,  also called
$A_\infty$-algebras, are dg vector spaces $(A,d)$ together with operations
\hbox{$\{\mu_k \! :\! A^{\otimes k} \to A\}_{k \geq 2}$} 
such that \hbox{$\mu_2\! :\! A \ot\! A \to A$} is
associative up to the chain homotopy~$\mu_3$, and
$\rada{\mu_2}{\mu_{k-1}}$ satisfy, for each $k \geq
4$, a specific coherence relation up to the chain homotopy~$\mu_k$, 
cf.~\cite[Example~II.3.132]{markl-shnider-stasheff:book} for a precise
definition. $A_\infty$-algebras indeed
behave well with respect to chain homotopy equivalences of their
underlying dg vector spaces~\cite{tr}.

Operad theory explains why it is so.  In the operadic
context, dg associative algebras emerge as algebras for the
non-$\Sigma$ (non-symmetric) operad $\Ass$. 
It is a quadratic operad, and its quadratic dual
$\Ass^!$ is isomorphic to
$\Ass$~\cite[Thm.~2.1.11]{ginzburg-kapranov:DMJ94}.\footnote{The
  results and definitions cited here were formulated for
  $\Sigma$- (symmetric) operads, but their obvious non-$\Sigma$
  versions hold as well.} 
Definition~4.2.14 of~\cite{ginzburg-kapranov:DMJ94} thus interprets
\hbox{$A_\infty$-algebras} as
algebras for the dual
dg operad $\D(\Ass)$, the dg operad
defined as the operadic bar construction applied to
the suspended piece-wise linear
dual of $\Ass$~\cite[page~245]{ginzburg-kapranov:DMJ94}.
The dual dg operad is, by construction, free with a quadratic differential.  
In~particular, it is cofibrant, and 
its cofibrancy  guarantees good homotopy properties of its
algebras, which is a~principle latent  in topology since 
Boardman-Vogt's~\cite{boardman-vogt:73}, spelled out for algebra e.g.\
in~\cite{berger-moerdijk:02,markl:ha}.

Strongly homotopy algebras are algebraic cousins of
{\em $A_\infty$-spaces\/}, which are topological spaces
with an action of the  operad
$\K$ of Stasheff's associahedra~\cite{jds:hahI}. The operad~$\K$ lives in the
category of contractible cell complexes; its arity-two 
piece~$\K(2)$ is the point, $\K(3)$ is the interval and $\K(4)$ the pentagon.
A portrait of $\K(5)$ can be found
in~\hbox{\cite[page~10]{markl-shnider-stasheff:book}}. 

A bridge between topology and algebra is an explicit isomorphism
between the cellular chain complex $\CC_*(\K)$ of Stasheff's operad
and the dual dg operad $\D(\Ass)$, constructed
in~\hbox{\cite[Example~4.8]{markl:zebrulka}}.  The contractibility of
the pieces of $\K$ together with the isomorphism
$\D(\Ass) \cong \CC_*(\K)$ implies that the dg vector space
$\D(\Ass)(n)$ is acyclic in positive dimensions for each arity $n$.
One can easily derive  from this that the canonical map $\D(\Ass) \to \Ass$ is a
piece-wise homology isomorphism, therefore $\D(\Ass)$ is the minimal
model of
$\Ass$~\hbox{\cite[Def.~II.3.124]{markl-shnider-stasheff:book}}. This,
by~\cite[Prop.~2.6]{markl:zebrulka}, implies that $\Ass$ is Koszul.

Let us point out that Koszulity is an important property of quadratic operads that
e.g.~implies the existence of a well-behaved deformation theory for the related
category of algebras. The r\^ole of the contractibility of Stasheff's
associahedra in this context was noticed already in the author's 1994 
article~\cite{muj-super}.

Another distinguishing feature of  $\Ass$ is that it is the
linearization of the terminal non-$\Sigma$ operad in the category of
sets, i.e.\ of the $\Set$-operad
whose all pieces equal the one-point set~$*$ and all structure operations
are the identifications $* \times \cdots \times * = *$. Abusing
slightly the terminology, we will call $\Ass$ the terminal
non-$\Sigma$ operad as well.  The above observations can be then concisely
formulated as

\begin{fact}
The cellular chain complex of the Stasheff's
associahedron is the minimal model of the terminal non-$\Sigma$-operad.
This terminal operad is quadratic Koszul.
\end{fact}

The original humble aim of this note was to formulate and prove an
analog of this Fact for permutads, using indications provided 
by J.-L.~Loday and M.~Ronco in~\cite{loday11:_permut}. 
We however decided to extend it to
an advertisement for the theory of operadic categories developed in~\cite{Sydney}
and~\cite{duodel}, using permutads as an excuse. 

\vskip .2em
\noindent 
{\bf Part~1\/}, independent of the theory of operadic categories, 
starts by recalling,
following~\cite{loday11:_permut}, permutads and the related notions. We then
introduce quadratic permutads and their Koszul duals and formulate the
Koszulity property for quadratic permutads in terms of a suitably
defined dual bar construction.
We close Part~1 by a permutadic  analog of the 
Ginzburg-Kapranov power series test~\cite[Theorem~3.3.2]{ginzburg-kapranov:DMJ94}.
The results of the first part, namely
Proposition~\ref{Pozvani_do_MSRI} and
Theorem~\ref{Pan_Hladky_zase_del8_problemy.} combined with 
\cite[Proposition~5.4]{loday11:_permut}, give the following wished-for
permutadic analog of the Fact above:  

\begin{theoremA}
The cellular chain complex of the
permutohedron\footnote{Sundry definitions of the permutohedron are
assembled in Appendix~2 to \cite{loday11:_permut}.} 
is the minimal model of the terminal permutad.
This terminal permutad is quadratic Koszul.
\end{theoremA}

An important feature of the dual bar 
construction\footnote{I.e.\  the bar construction applied to the
linear dual.} $\D(A)$ of a
permutad $A$ introduced in Definition~\ref{Odhodlam_se_zitra_vyjet?}
is that it is again a permutad. This self-duality  
is not automatic. For instance, the dual bar
construction  of
a commutative associative algebra is a Lie algebra, the
dual bar construction of a modular operad is a twisted modular operad,
\&c.

An explanation of self-duality of permutads is offered
by~\cite[Definition~5]{dotsenko-khoroshkin} which presents them as
associative algebras in the category of (graded) vector spaces with an
unusual symmetric monoidal
structure~(\ref{Vcera_byla_slava_na_Narodni.}), 
while the self-duality of
associative algebras is classical. We however give an alternative
explanation that uses the theory of operadic categories, which are
the subject of 

\vskip .2em
\noindent 
{\bf Part~2.}
As shown in \cite[\S14.4]{Sydney}, there exists an operadic category
$\ttP$ such that permutads are algebras, in the
sense of~\cite[Definition~1.20]{duodel}, over the terminal
$\ttP$-operad $\Pterm$; we formulate this statement with  a slightly
different proof 
as Proposition~\ref{Treti_den_je_kriticky.}. The self-duality of the category of
permutads follows from
Theorem~14.4 of~\cite{Sydney} which says that the $\ttP$-operad
$\Pterm$ governing permutads 
is binary quadratic and self-dual. We will analyze this
phenomenon in the context of general operadic categories in our future work.
We complement these results by proving:

\begin{theoremB}
The terminal $\ttP$-operad $\Pterm$ is Koszul.
\end{theoremB}
As in the case of `classical' operads, Koszulity of  $\Pterm$ leads to
an effective description of its minimal model $\min$ which we give
in the proof of Theorem~\ref{Vcera_jsem_se_vratil_z_Ciny.}. 
Algebras over $\min$ are, due to the
philosophy pioneered in \cite{markl:zebrulka}, 
strongly homotopy permutads. Their explicit description in terms of
operations and axioms is
given in Proposition~\ref{Klasicka_poprijezdova_deprese}.

\vskip .2em
\noindent
{\bf Where the associahedron hides?}
The present article shall also illustrate the potential
of Grothendieck's construction in
operadic categories and the related discrete
opfibrations~\hbox{\cite[Subsect.~5.2]{Sydney}}. 
Let $\Ord$ be the operadic category of finite ordinals and their
order-preserving surjections; recall that $\Ord$-operads are ordinary
constant-free non-$\Sigma$-operads.
It turns out that the operadic category $\ttP$ is the
Grothendieck's construction applied to a~certain $\Ord$-cooperad $\oC$. 
One thus has a discrete opfibration
\hbox{$\des: \ttP \to \Ord$}. 

A consequence is that the restrictions  along $\des :\ttP \to \Ord$ of
free non-$\Sigma$-operads are free $\ttP$-operads.  In particular,
the restriction $\des^*(\K)$ of Stasheff's operad of the
associahedra turns out to be  the
convex polyhedral realization of the minimal model $\min$ of the
terminal $\ttP$-operad $\Pterm$, cf.~Remark~\ref{V_Koline_jsem_povesil_lustr.}.

In Theorem~\ref{Je_teprve_utery_a_uz_padam_na_usta.}
we prove that the restriction along 
\hbox{$\des: \ttP \to \Ord$} brings Koszul
\hbox{non-$\Sigma$} operads into Koszul $\ttP$-operads. This is a~particular
case of an important feature of opfibrations between operadic
categories, and an advertisement for our future work.

\vskip .5em
\noindent 
{\bf Conventions.}
Our background monoidal category will be the category $\Vect$ of
differential graded, or dg for short, vector
spaces $V = \bigoplus_{k \in {\mathbb Z}} V_k$ over a 
field $\bfk$ of characteristic $0$; the preferred degree of
differentials will be $-1$. The linear duals are taken
degree-wise, i.e.\ the degree $k$ component of the  
dual $V^*$ of a graded space above will be $\Vect(V_k, \bbk)$,
$k \in {\mathbb Z}$.  

If not stated
otherwise, all algebra-like objects (monoidal categories,
permutads) will be nonunital. Operadic categories, their operads and
algebras over these operads were introduced in
\cite[\S I.1]{duodel}.
The standard reference for `classical' operads, quadratic duality and
Koszulness is~\cite{ginzburg-kapranov:DMJ94} or more
recent~\cite{loday-vallette} or~\cite{markl-shnider-stasheff:book}.

The cobar construction of a linear dual of an algebra or permutad will
be called the dual bar construction, while the cobar construction
applied to the linear dual of an operad will be called,
following~\cite{ginzburg-kapranov:DMJ94}, the dual dg operad.

\vskip .3em
\noindent
{\bf Acknowledgment.} 
I wish to express my gratitude to Vladimir Dotsenko for presenting permutads 
in a broader context to me  and turning my attention
to~\cite{dotsenko-khoroshkin}. 
Also suggestions of the referee
lead to a substantial improvement of the paper.

\part{Classical approach to permutads}

\section{Permutads -- recollections}

Permutads appear under various names in various disguises. 
Following~\cite{dotsenko-khoroshkin} and~\cite{loday11:_permut} we recall 
some of them in this section. The most
concise definition uses the skeletal category $\Fin$ of finite
non-empty sets with objects the ordinals 
$\underline n  := \{\rada 1n\}$, $n \geq 1$, and its subcategory
$\Surj$ of surjections. For any order-preserving surjection $f: A \to
B$ of finite ordered sets there exists a unique $\alpha : \underline n
\to \underline k
\in \Surj$ in the commutative diagram
\begin{equation}
\label{S_Jarkou_do_Mnisku.}
\xymatrix@C=3.5em@R=2em{A\ar[r]^f\ar[d]_\cong &B\ar[d]^\cong
\\
\underline n \ar[r]^\alpha&  \underline k
}
\end{equation}
where the vertical arrows are order-preserving isomorphisms. In this
sense, every order-preserving surjection $f: A \to
B$ of finite ordered sets may be {\em interpreted\/} as a morphism in $\Surj$.
For $r : \underline m \epi \underline n \in \Surj$ and $i \in \underline n$ 
we denote by $\inv r (i)$ the pullback
\[
\xymatrix@C=3em{\inv r (i) \ar[r] \ar[d] & \underline m \ar[d]^r
\\
\underline 1 \ar[r]^{1\, \mapsto\, i} & \underline n
}
\]
in $\Fin$. Notice that $\inv r (i) = \underline c$, with
$c\geq 1$ the cardinality of the set-theoretic preimage of $i$ under~$r$.
We emphasize that  $\inv r (i)$ {\em does not\/} denote the set-theoretic
inverse image, but the categorial pull-back, and serves as an example
of a fiber in an operadic category. 
Permutads live in the category 
$\Coll$ of collections 
\[
A = \{A(\underline n) \in \Vect\ | \
\underline n \in \Fin\},
\] 
which we consider as a (non-unital) 
monoidal category with the 
shuffle product~\cite[Definition~1]{dotsenko-khoroshkin}
\begin{equation}
\label{Vcera_byla_slava_na_Narodni.}
(A \odot B)(\underline n) := \bigoplus_{r \in \Surj(\underline n, \underline
  2)} \big\{ A(\inv r (1)) \ot  A(\inv r (2)) \big\}.
\end{equation}

The coproducts of graded vector spaces obviously commute 
with $\odot$ from both sides,
which is the
property that, according to~\cite{barr}, guarantees that the monoidal
category $\Coll$ behaves in
most aspects as the standard monoidal category of graded vector spaces.
The following definition is taken from~\cite{dotsenko-khoroshkin}.

\begin{definition}
\label{Vcera_jsem_prijel_z_Luminy.}
A {\em permutad\/}, also called a {\em shuffle algebra\/}, is a monoid
for the monoidal product~\eqref{Vcera_byla_slava_na_Narodni.}.
\end{definition}

\begin{remark}
The term shuffle algebra comes from the fact that surjections 
$r$ in~\eqref{Vcera_byla_slava_na_Narodni.} 
are in one-to-one
correspondence with $(n_1,n_2)$-shuffles for some $n_1 \+ n_2 =
n$. Let us recall what shuffles are.
Given natural numbers  $\Rada n1k \geq 1$ such that $n_1 \+\cdots \+ n_k
=n$, an {\em $(\Rada n1k)$-shuffle \/}  is a permutation $\upsilon \in
\Sigma_n$ such that for any $j = \rada1k$ one has 
\begin{equation}
\label{karantena}
\upsilon(n_1\+ \cdots \+n_{j-1} + 1) < \upsilon(n_1\+ \cdots \+n_{j-1}
+ 2) < \cdots < \upsilon(n_1\+ \cdots \+n_j).
\end{equation} 
It is clear that $(\Rada n1k)$-shuffles are in one-to-one correspondence with 
ordered partitions
$[I_1|\cdots| I_k]$ of $\underline n = \{\rada 1n\}$ into $k$ disjoint nonempty
subsets, with $I_j$ being, for $1 \leq j \leq k$, the subset 
of $\underline n$ consisting of the
numbers in~\eqref{karantena}. When convenient, we use partitions to denote
shuffles. Thus, for  instance, $[2,4|1,3,6|5,7,8]$ 
denotes the $(2,3,3)$-shuffle
\[
(1,2,3,4,5,6,7,8) \longmapsto (2,4,1,3,6,5,7,8) \in \Sigma_8.
\]
Denoting by $\ush(\Rada n1k) \subset \Sigma_n$ the set of all $(\Rada
n1k)$-shuffles, one has the isomorphism
\begin{equation}
\label{okynko}
\Surj(\underline n, \underline k) \cong \bigsqcup_{n_1 +\cdots + n_k
=n} \ush(\Rada n1k) 
\end{equation}
where $\bigsqcup$ denotes the disjoint union, which
identifies $r \in \Surj(\underline n, \underline k)$ with the partition
$[I_1|\cdots| I_k]$ given by
\[
I_j := \{i \ |\  r(i) =j   \}, \ 1 \leq j \leq k.
\]
For the future use we denote
\begin{equation}
\label{Za_chvili_volam_Jarce.}
r_j = \card  \{i \ |\  r(i) =j   \},  \ \hbox { and }\
\varepsilon(r) := \sgn(\upsilon_r) \in \{-1,+1\},\ 1 \leq j \leq k,
\end{equation}
where $\upsilon_r \in \Sigma_n$ is the shuffle corresponding to $r$ in the
correspondence~\eqref{okynko}. 
\end{remark}

We are going to give two alternative, explicit definitions of
permutads. The first one emphasizes their combinatorial nature, the
second one is tailored for the purposes of our paper. Let us
introduce necessary auxiliary notation.
We will need
\[
\pi_{(1,2+3)}:  \ush(n_1,n_2,n_3) \longrightarrow  \ush(n_1,n_2\+n_3)
\]
given by
$\pi_{(1,2+3)}([I_1|I_2|I_3]) := [I_1|I_2 \cup I_3]$, and the
map 
\[
\pi_{(1+2,3)}:  \ush(n_1,n_2,n_3) \longrightarrow  \ush(n_1\+ n_2,n_3)
\]
given by the obvious similar formula. We also define the map
\[
\pi_{(1,2)}:  \ush(n_1,n_2,n_3) \longrightarrow  \ush(n_1,n_2)
\]
by $\pi_{(1,2)}([I_1|I_2|I_3]) := [\stan(I_1)| \stan(I_2)]$, where the
standardization $\stan : I_1 \cup I_2 \to \{1,\ldots,n_1\+n_2\}$ is the unique
monotonous map. The map 
\[
\pi_{(2,3)}:  \ush(n_1,n_2,n_3) \longrightarrow
\ush(n_2,n_3)
\]  
is likewise given by  
$\pi_{(2,3)}([I_1|I_2|I_3]) := [\stan(I_2)| \stan(I_3)]$. 
For example
\begin{align*}
\pi_{(1,2+3)}([2,4|1,3,6|5,7,8]) &= [2,4|1,3,5,6,7,8],  
\\
\pi_{(1+2,3)}([2,4|1,3,6|5,7,8]) &= [1,2,3,4,6|5,7,8],  
\\
\pi_{(1,2)}([2,4|1,3,6|5,7,8]) &= [2,4|1,3,5], \hbox { and}
\\
\pi_{(2,3)}([2,4|1,3,6|5,7,8]) &= [1,2,4|3,5,6].    
\end{align*}
Notice that the above maps induce isomorphisms in the diagram
\begin{equation}
\label{Konci mi karantena.}
\xymatrix{
& \ush(n_1,n_2,n_3)
\ar[rd]^{\ \ \ \ \  \pi_{(1,2)} \times \pi_{(1+2,3)}}_\cong
\ar[ld]_{ \pi_{(2,3)} \times \pi_{(1,2+3)} \ \ \ \ \  }^\cong&
\\
\ush(n_2,n_3) \times \ush(n_1,n_2\+n_3)
&&\ \ush(n_1,n_2) \times \ush(n_1 \+n_2,n_3).
}
\end{equation}
We are ready to give the following explicit

\begin{definition}
\label{Vcera jsem pichnul dusi.}
A {\em permutad\/} is a collection $A \in \Coll$ together with operations
\begin{equation}
\label{Zase jsem to poplet s tim casopisem.}
\bullet_{\tau} : A({\underline n}_{\, 1}) \ot A(\underline n_{\, 2}) 
\longrightarrow A(\underline n)
\end{equation}  
given for each shuffle $\tau \in
\ush(n_1,n_2)$ with $n_1,n_1 \geq 1$, $n=n_1 \+ n_2$.
These operations are required to satisfy 
\[
 \bullet_\gamma(\id \ot \bullet_\delta) =
 \bullet_\lambda(\bullet_\sigma \ot \id)
 \tag{$\shufflerm$}
\]
for each $\shufflerm \in \ush(n_1,n_2,n_3)$, $n_1,n_2,n_3 \geq 1$,  
with   $\gamma = \pi_{(1,2+3)}(\shufflerm)$,
$\delta = \pi_{(2,3)}(\shufflerm)$,
$\lambda = \pi_{(1+2,3)}(\shufflerm)$ and
$\sigma = \pi_{(1,2)}(\shufflerm)$. 
\end{definition}

In elements, equation ($\shufflerm$) can of course be written as 
\begin{equation}
\label{Zitra se vykoupu.}
a \bullet_\gamma (b \bullet_\delta c) = (a \bullet_\sigma b) \bullet_\lambda c,
\end{equation}
for $a \in  A({\underline n}_{\, 1})$,
$b \in  A({\underline n}_{\, 2})$ and
$c \in  A({\underline n}_{\, 3})$.
To rewrite Definition~\ref{Vcera jsem pichnul dusi.} into 
the form closer to~\hbox{\cite{loday11:_permut}}, we put, 
using the isomorphisms in~\eqref{Konci mi karantena.},
\begin{equation}
\label{Projdu prohlidkou?}
(\id \times \delta) \cdot \gamma :=  
[\pi_{(2,3)} \times \pi_{(1,2+3)}]^{-1}(\delta \times \gamma),\ 
(\sigma \times \id) \cdot \lambda 
 : = [\pi_{(1,2)} \times \pi_{(1+2,3)}]^{-1}(\sigma \times \lambda).
\end{equation}
Shuffle algebras then appear as structures with operations~\eqref{Zase
  jsem to poplet s tim casopisem.} that satisfy (\shufflerm) for each
$\delta,\gamma,\sigma,\lambda$ such that $(\id \times \delta) \cdot
\gamma = (\sigma \times \id) \cdot \lambda$ in
$\ush(n_1,n_2,n_3)$. This is precisely the
definition given in~\hbox{\cite[Subsect.~3.4]{loday11:_permut}}.

\begin{example}
Assume that 
 $A = \bigoplus_{n \geq 1}A_n$ is a graded associative algebra 
with a product $\bullet: A \ot A \to A$. If we define 
\[
A(\underline n) := A_n, \ n \geq 1,\  \hbox { and } \
\bullet_\tau := \bullet \ \hbox { for each shuffle $\tau$,}
\]
then~\eqref{Zitra se vykoupu.} follows from the associativity of
the product. Thus graded associative algebras provide 
examples of permutads. Other examples can be found 
in~\cite[Subsect.~2.2]{dotsenko-khoroshkin} or in the second half of~\cite{loday11:_permut}.
\end{example}

Let us translate Definition~\ref{Vcera jsem pichnul dusi.} into the
language of surjections.  
For surjections $t : \underline m \to \underline 2$ and 
$s : \inv t(2) \to \underline 2$ we denote by $t(\id \times s ) :
\underline m \to \underline 3$ the surjection defined by the
commutativity of the diagrams
\begin{equation}
\label{Zitra jdu na prohlidku.}
\xymatrix@C=3em{\inv t (1) \ar[r] \ar[d] & \underline m \ar[d]^{t(\id
    \times s)}
\\
\ \underline 1\  \ar@{^{(}->}[r]^{\iota_1}& \underline 3
}
\raisebox{-1.8 em}{\hbox {\ \  and\ \ } }
\xymatrix@C=3em{\inv t (2) \ar[r] \ar[d]_s & \underline m \ar[d]^{t(\id
    \times s)}
\\
\ \underline 2\  \ar@{^{(}->}[r]^{\iota_{2,3}}& \underline 3
}
\end{equation}
where $\iota_1(1) := 1$, $\iota_{2,3}(1) : = 2$ and $\iota_{2,3}(2) :=
3$. More explicit, though less
conceptual, description of  $t(\id \times s )\in \Surj(\underline m,
\underline 3)$, goes as follows. 
For $t: \{\rada 1m\} \to \{1,2\}$, let \hbox{$g_1 < \cdots <
g_a$} be numbers such that 
\[
\{\Rada g1a\} = \{j \ |\ t(j) = 2\}  \subset \{\rada 1m\},
\]
so that  $t^{-1}(2) = \underline a$ and   
$s : \{\rada 1a\} \to \{1,2\}$. Then $t(\id \times s ) :  \{\rada
1m\}  \to \{1,2,3\}$ is given by
\[
t(\id \times s )(j) := 
\begin{cases}
1, & \hbox {if $j \not\in  \{\Rada g1a\}$},
\\
2, & \hbox {if $j = g_k$ for some $k$ with $s(k) = 1$, and}
\\
3, & \hbox {if $j = g_k$ for some $k$ with $s(k) = 2$.}
\end{cases}
\]
Analogously we define, for $u \in
\Surj(\underline m,\underline 2)$ and 
$v : \inv u(1) \to \underline 2$, the surjection $u(v \times \id) \in \Surj(\underline m,\underline 3)$ as follows.
Let $h_1 < \cdots < h_b$ be such that 
\[
\{\Rada h1b\} = \{j \ |\ u(j) = 1\}  \subset \{\rada 1m\},
\]
so $v : \{\rada 1b\} \to \{1,2\}$. Then
\[
u(v \times \id)(j) := 
\begin{cases}
1, & \hbox {if $j = h_k$ for some $k$ with $v(k) = 1$.} 
\\ 
2, & \hbox {if $j = h_k$ for some $k$ with $v(k) = 2$, and}
\\
3, & \hbox {if $j \not\in  \{\Rada h1b\}$}.
\end{cases}
\]
We leave as an exercise to express $u(v \times \id)$ in the diagrammatic
language of~\eqref{Zitra jdu na prohlidku.}. 

If the surjection $t$ above corresponds, under 
isomorphism~\eqref{okynko}, to a shuffle $\gamma$,
$s$ corresponds to $\delta$, $u$ to $\lambda$ and $v$ to $\sigma$, then  
$t(\id \times s)$ corresponds to $(\id \times \delta) \cdot \gamma$ 
in~\eqref{Projdu prohlidkou?},
and $u(v \times \id)$ to $(\sigma \times \id)\cdot \lambda$.
With these definitions at hand we may formulate the following

\begin{definition}
\label{V_nedeli_letim_do_Moskvy.}
A {\em permutad\/} is a collection $A \in \Coll$ together with operations
\begin{equation}
\label{Pozitri_jedem_do_Zvanovic.}
\pcirc_r : A(\inv r(1)) \ot  A(\inv r(2)) \longrightarrow A(\underline n)
\end{equation}  
defined 
for each surjection $r \in \Surj(\underline n,\underline 2)$.
These operations shall satisfy 
\begin{equation}
\label{Vratim_se_z_te_Ciny?}
\pcirc_u(\pcirc_v \ot \id) = \pcirc_t(\id \ot \pcirc_s)
\end{equation}
for each $m \geq 1$ and surjections $u,t : \underline m \to \underline
2$,  $s : \inv t(1) \to \underline 2$ and $v : \inv u(2) \to \underline 2$ 
such that 
\[
t(\id \times s) = u(v \times \id).
\]
\end{definition}

Definitions~\ref{Vcera jsem pichnul dusi.} 
and~\ref{V_nedeli_letim_do_Moskvy.} suggest that permutads are
close in nature to operads. One thus sometimes uses, as
in \cite{loday11:_permut},
a shifted grading 
$A_{n+1} := A(\underline n)$, $n \geq 1$. The underlying 
collection $A_2,A_3,\ldots$ of a permutad is then the same as the
underlying collection of
a $1$-connected non-unital non-symmetric \hbox{(non-$\Sigma$)}~operad.
 
Let us  recall from \cite{loday11:_permut} a monadic definition of
permutads. For a~collection $B \in \Coll$, define $\PP(B) \in \Coll$ 
by $\PP(B) :=
\bigoplus_{k \geq 1} \PP^k(B)$, where
\begin{equation}
\label{Zitra_jedu_do_Zvanovic.}
\PP^k(B)(\underline n) :=
\bigoplus_{r \in
  \Surj(\underline n,\underline   k)} \bigotimes_{i\in \underline k}
B(\inv r(i)).
\end{equation}

\begin{remark}
Assume that $A$ is a permutad as in Definition \ref{V_nedeli_letim_do_Moskvy.}.
Since clearly $\PP^1(A) \cong A$ and
\[
\PP^2(A) \cong \bigoplus_{r \in
  \Surj(\underline n,\underline   k)} A(\inv r(1)) \ot  A(\inv
r(2)),
\] 
structure operations~\eqref{Pozitri_jedem_do_Zvanovic.}
assemble~into one map
\begin{equation}
\label{eq:Petr_v_Praze}
\pcirc :  \PP^2(A) \longrightarrow \PP^1(A). 
\end{equation}
\end{remark}

As shown in~\cite[Section~1.6]{loday11:_permut}, 
there exists a natural morphism of collections
\begin{equation}
\label{Pujdu_si_nakoupit.}
\Gamma_B : \PP(\PP(B)) \longrightarrow \PP(B), \ B \in \Coll,
\end{equation}
which, together with the obvious inclusion $\iota_B : B
\hookrightarrow \PP(B)$, makes $\PP(-)$ a monad in $\Coll$. One has
the fourth

\begin{definition}
\label{Posledni_moznost_jet_na_kole_a_ja_mam_chripku.}
A {\em permutad\/} is an algebra for the monad $\PP: \Coll \to \Coll$.
\end{definition}

As a consequence of classical statements, see
e.g.~\cite[Theorem~VI.2.1]{maclane:71}, $\PP(B)$ is the {\em free
permutad\/} generated by $B$. 

\begin{example}\hskip -.5em\footnote{Quoting a Czech physicist: ``Old men give
    only good examples, since they are unable to give bad ones.''}
\label{Byl_jsem_si_obehnout_dvorek.}
An important r\^ole will be featured by the permutad $\term$ with
\hbox{$\term(\underline n) := \bfk$} for all $\underline n \in \Fin$, with all
operations~\eqref{Zase jsem to poplet s tim casopisem.} the canonical
isomorphisms $\bbk \ot \bbk \cong \bbk$. It is the linearization of 
the terminal permutad in the
category $\Set$ of sets,
so we will call it, being aware of slight abuse of terminology, the
{\em terminal permutad\/}.
\end{example}

\begin{example}
\label{Za_tyden_jedu_na_Zimni_skolu.}
The permutohedron $P_{n+1}$ is, for $n \geq 1$, a convex polytope whose 
faces are labelled by planar rooted trees with levels, see
e.g.~\cite{tonks97} or \cite[Appendix~2]{loday11:_permut}. Its  cellular chain complex
\[
\CC_*(P) = \{\CC_*(P_{n+1})\ | \
\underline n \in \Fin\}
\] 
is a permutad in the category of dg vector
spaces whose underlying permutad is free, generated by the collection
$B$ which has
\begin{equation}
\label{Je_streda_a_melu_z_posledniho.}
B(\underline n)   :=  \Span(c_{n-1}) \cong  \ \uparrow^{n-1} \bbk,\ n \geq 1,
\end{equation}
the ground field $\bbk$ placed in degree $n\!-\!1$, 
cf.~\cite[Theorem~5.3]{loday11:_permut}. In~\eqref{Je_streda_a_melu_z_posledniho.},
$c_{n-1} := \uparrow^{n-1}\! 1$, $1 \in \bbk$, is the generator of
$\uparrow^{n-1} \bbk$.   Indeed,
formula~\eqref{Zitra_jedu_do_Zvanovic.} gives
\[
\PP(B)(\underline n) \cong \bigoplus_{1 \leq k \leq n} \  \bigoplus_{r \in
  \Surj(\underline n,\underline k)}
\hskip -.5em
\Span(
c_{r_1-1} \ot \cdots \ot c_{r_k-1}) 
\cong
\bigoplus_{1 \leq k \leq n} \  \bigoplus_{r \in
  \Surj(\underline n,\underline k)}
\uparrow^{n-k} \bfk,
\]
where $r_i$ is the cardinality of the set-theoretic preimage of $i \in
\underline k$ via $r : \underline n \epi \underline k$.
The set
$\Surj(\underline n,\underline k)$ is isomorphic to the set of
ordered decompositions of $\{\rada 1n\}$ into $k$ disjoint non-empty
subsets which is, in turn, isomorphic to the set of planar rooted
trees with $n\+1$ leaves and $k$ levels, see
e.g.~\cite[\S1.3]{biperm}. 

Explicitly, the isomorphism $\CC_*(P) \cong \PP(B)$
identifies the `big' $n$-dimensional cell of $P_n$ with the generator
$c_n$ of $\PP(B)$, $n \geq 0$. A formula for the boundary operator
of $\CC_*(P)$ written in terms of $\PP(B)$ can be found in
\cite[\S9.3]{loday11:_permut}. We will return to the permutohedron in
Proposition~\ref{Pozvani_do_MSRI}. 
\end{example}

\begin{example}
Each permutad $A$ admits a {\em permutadic suspension\/} $\ss A$. Its
underlying collection is $\ss A(\underline n) : =
\hbox{$\uparrow^n \hskip -.3em A(\underline n)$}$, $n \geq 1$, and the structure operations~\eqref{Pozitri_jedem_do_Zvanovic.}
\[
\sspcirc :\ss A(\inv r(1)) \ot \ss A(\inv r(2)) \longrightarrow  \ss
A(\underline n),
\ r \in \Surj(\underline n,\underline 2),
\]
are defined by
\[
(\uparrow^{r_1} \hskip -.3em a_1) \, \sspcirc  (\uparrow^{r_2}\hskip
-.3em a_2) : =     
\varepsilon(r) \cdot (-1)^{r_2(r_1 + \deg(a_1))}  \uparrow^{n}   \hskip -.3em
\hbox{$(a_1\, \pcirc_r\, a_2)$},
\]
in which $\pcirc_r$ is the structure
operation of $A$, $a_i \in A(\inv
r(i))$, $i = 1,2$, and $r_1, r_2$ and $\varepsilon(r)$ 
are as in~\eqref{Za_chvili_volam_Jarce.}.
\end{example}

\begin{remark}
There exist symmetric versions
of permutads, called {\em twisted associative algebras\/}, whose
underlying collections bear actions of the symmetric groups,
see~\cite[Chapter~4]{BrDo} for their definition. 
Koszul duality for these objects is treated in a recent
preprint~\cite{Kho}. We believe that the methods developed in the present
article can be easily modified also to twisted associative algebras.
\end{remark}

\section{Koszul duals of permutads}
\label{2+3}

Most of the notions in this section and in the first half of 
Section~\ref{Pristi_tyden_mam_spoustu_uradovani.}
are straightforward translations of the analogous standard 
notions for associative algebras
given by replacing the usual ambient monoidal structure of
vector spaces by the shuffle
product~(\ref{Vcera_byla_slava_na_Narodni.}). Our preferred form of
structure operations of
permutads will be that in
Definition~\ref{V_nedeli_letim_do_Moskvy.} based on surjections.
Let us start by recalling 
the following definition of \cite[\S4.4]{loday11:_permut}.

\begin{definition}
\label{Zavola_Petr?}
A permutad $A$ is {\em quadratic\/} if it is of the form $\PP(B)/(S)$,
where $(S)$ is the permutadic ideal generated by a subspace $S
\subset \PP^2(B)$. Such a quadratic permutad $A$ is {\em binary\/},
if  $B(\underline n) = 0$
for $n \not= 1$.  
\end{definition}

\begin{example}
\label{Zitra_sraz_s_Jitkou_v_ustavu.}
Let us denote, abusing the notation again, by $\PP(\mu)$ the free
permutad generated by one generator $\mu \in
\PP(\mu)(\underline 1)$. As
stated in \cite[Corollary~4.7]{loday11:_permut}, the terminal permutad
$\term$ recalled in Example~\ref{Byl_jsem_si_obehnout_dvorek.} is binary
quadratic, 
\begin{equation}
\label{Mam_pozvani_do_Bonnu_na_kveten.}
\term \cong \PP(\mu)/(S), 
\end{equation}
with $S \subset \PP^2(\mu) = \PP(\mu)(\underline 2)$ the $\bbk$-linear span
\[
S := \Span\big\{
\pcirc_{\id}(\mu \ot \mu) - \pcirc_{(21)}(\mu \ot \mu) \big\},
\] 
where $\id : \underline 2 \to \underline 2$ the identity and  $(21)
: \underline 2 \to \underline 2$ the transposition $1,2 \mapsto
2,1$. The verification of~\eqref{Mam_pozvani_do_Bonnu_na_kveten.} is simple. 
Formula~\eqref{Zitra_jedu_do_Zvanovic.} immediately
gives that
\[
\PP(\mu)(\underline n) \cong \Span (\Sigma_n),
\]
the $\bbk$-linear span of the symmetric group of $n$ elements, while
modding out by the ideal~$(S)$ identifies any two permutations in
$\Sigma_n$ that differ
by transposition of adjacent elements. Since transpositions act
transitively, one gets  $\PP(\mu)/(S)(\underline n) \cong \bfk \cong
\term(\underline n)$, $n \geq 1$, as~expected.
\end{example}

Every binary quadratic permutad $A$ as in Definition~\ref{Zavola_Petr?} 
possess its {\em Koszul dual\/} 
$A^!$, which is a binary quadratic permutad defined as
\[
A^!:=\PP(\susp B^*)/(S^\perp), 
\]
where $\susp B^*$ is the suspension of the  component-wise linear 
dual of the generating
collection~$B$, and $S^\perp \subset \PP^2(\susp B^*)$ is the annihilator
of $S \subset \PP^2(B)$ with respect to the obvious natural degree~$-2$ pairing between
\[
\PP^2(\susp B^*) = \bigoplus_{r \in
  \Surj(\underline 2,\underline   2)}
\hskip  -.5em
 \susp B(\underline 1)^*\, \ot  
\susp B(\underline 1)^*
\cong \
\susp^2
\bigoplus_{\sigma \in
\Sigma_2}
 B(\underline 1)^*\, \ot  
 B(\underline 1)^*
\]
and
\[
\PP^2( B) = \bigoplus_{r \in
\Surj(\underline 2,\underline   2)} 
\hskip  -.5em
B(\underline 1)\, \ot  
B(\underline 1)
\cong
\bigoplus_{\sigma \in
\Sigma_2}
 B(\underline 1)\, \ot  
B(\underline 1)
\]

\begin{proposition}
\label{Dnes_vecer_prijde_Jarka.}
The terminal permutad $\term$ is Koszul self-dual in the sense that
its Koszul dual $\term^!$ is isomorphic to its  permutadic suspension  
$\ss\term$. Explicitly,
\[
\term^!(\underline n) \cong\ \susp^{\hskip .15em n}\,\bfk,\ n \geq 1,
\] 
with the structure operation~\eqref{Pozitri_jedem_do_Zvanovic.}
the composition
\[
\term^!(\underline r_1) \ot \term^!(\underline r_2) \cong
\hbox{$\susp^{\hskip .15em r_1} \hskip .1em \bbk$} \ot \hbox{$\susp^{\hskip .15em r_2} 
\hskip .1em \bbk$}\hskip .1em \cong\ \hbox{$\susp^{\hskip .15em
  n}\,\bbk$}
\cong \term^!(\underline n)
\]
multiplied with $\varepsilon(r) \in \{-1,+1\}$, where $r_1$, $r_2$ and
$\varepsilon(r)$ have the same meaning as in~\eqref{Za_chvili_volam_Jarce.}.
\end{proposition}

\begin{proof}
By the definition of the Koszul dual, one has
\[
\term^! \cong \PP(\mu^\uparrow)/(S^\perp), 
\]
where  $\mu^\uparrow \in \PP(\mu^\uparrow)(\underline 1)$ is a degree
$1$ generator and $S^\perp$ the span
\[
\Span\{
\pcirc_{\id}(\mu^\uparrow \ot \mu^\uparrow) +
\pcirc_{(21)}(\mu^\uparrow \ot \mu^\uparrow)\}.
\] 
By formula~\eqref{Zitra_jedu_do_Zvanovic.},
\[
\PP(\mu^\uparrow)(\underline n) \cong \  \uparrow^n    \Span (\Sigma_n),
\]
while
modding out by the ideal $(S^\perp)$ introduces the relation 
$\sigma' \sim -\sigma''$ for each 
$\sigma' \in \Sigma_n$ and $\sigma''$ obtained from $\sigma'$ by
a transposition of two adjacent elements. Therefore the assignment  
\[
\uparrow^n \bfk \ni \hskip .1em \uparrow^n \! 1 \longmapsto \id_n \in \Sigma_n
\]
leads to an isomorphism
$\uparrow^n\! \bfk \cong \PP(\mu^\uparrow)/(S^\perp)(\underline n)$.
The advertised formula for the structure operations easily follows from the
description of the operations in the free permutad $\PP(\mu^\uparrow)$
given in~\hbox{\cite[Section~1.6]{loday11:_permut}}.
\end{proof}

\section{Koszulity of permutads}
\label{Pristi_tyden_mam_spoustu_uradovani.}

We start this section with a permutadic version of the cobar
construction and the related dual bar construction.
We then establish that the dual bar
construction of the Koszul dual $\term^!$ of the terminal permutad $\term$ is isomorphic
to the cellular chain complex of the permutohedron. Finally, we prove
that $\term$ is Koszul and formulate a test for Koszulity of
permutads. The first definition is a 
harmless formal dual of Definition~\ref{V_nedeli_letim_do_Moskvy.}.

\begin{definition}
A {\em copermutad\/} is a collection  $C \in \Coll$ together with operations
\begin{equation}
\label{Strasliva_zima.}
\delta_r : C(\underline n) \longrightarrow C(\inv r(1))\! \ot\!
C(\inv r(2)),\
r \in \Surj(\underline n,\underline 2), \ n \geq 2,
\end{equation}  
satisfying the obvious dual versions of axioms~\eqref{Vratim_se_z_te_Ciny?}.
\end{definition}

\begin{example}
\label{Pisu_v_Koline_po_skoleni_ze_stavby.}
Assume that $A$ is a permutad such that $A(\underline n)$ is, for each
\hbox{$n \geq 1$}, either
finite-dimensional, or non-negatively or non-positively graded 
vector space of finite type.  Then its component-wise 
linear dual $A^* = \coll {A^*}n$
is a copermutad. This is in particular satisfied 
if $A$ is binary quadratic as in
Definition~\ref{Zavola_Petr?}, with $B(1)$ finite-dimensional.
\end{example}

\begin{definition}
A degree $s$ {\em derivation\/} of a permutad $A$ is a degree $s$ linear map $\varsigma
: A \to A$ of collections such that
\[
\varsigma\,\pcirc_r = \pcirc_r (\varsigma \ot \id) +  \pcirc_r(\id \ot
\varsigma),\
r \in \Surj(\underline n,\underline 2),\ n \geq 2,
\]
for the structure operations $\pcirc_r$
in~\eqref{Pozitri_jedem_do_Zvanovic.}.
In elements,
\[
\varsigma(a_1 \pcirc_r a_2) = \varsigma (a_1)\,  \pcirc_r\, a_2 +
(-1)^{s\cdot \deg(a_1)}
a_1\, \pcirc_r\, \varsigma(a_2), \ a_i \in A(\inv r(i)),\ i =1,2.
\]
\end{definition}

One easily verifies that each 
degree $s$ map of collections $\sigma :B \to \PP(B)$ uniquely
extends to a degree $s$ derivation $\varsigma$  of the free permutad
$\PP(B)$ satisfying $\varsigma|_B = \sigma$.
Let $C$ be a copermutad. 
Notice that~\eqref{Strasliva_zima.}
assemble into one map
\[
\delta : C \cong \PP(C) \longrightarrow \PP^2(C)
\]
thus one has a degree $-1$ map $\sigma :\ \desusp C  \to  \PP^2(\desusp C)$
defined as the composition
\[
\sigma :=\
\desusp C \stackrel \uparrow\longrightarrow C
\stackrel\delta\longrightarrow
\PP^2(C) \stackrel\cong\longrightarrow \,\uparrow^2 \PP^2(\desusp C)
\stackrel{\downarrow^2}\longrightarrow \PP^2(\desusp C)
\hookrightarrow  \PP(\desusp C)
\]
with \hbox{$\PP^2(C) \stackrel\cong\longrightarrow \,\uparrow^2\!
  \PP^2(\desusp C)$}  the
obvious canonical isomorphism.
Denote finally 
by $\pa_\cobar$ the unique extension of $\sigma$ into a degree $-1$
derivation of \hbox{$\PP(\desusp \hskip -.05em C)$}. 
One may verify by direct calculation that $\pa_\cobar^{\,2} =
0$, thus the following definition makes sense.

\begin{definition}
\label{Odhodlam_se_zitra_vyjet?}
The {\em cobar construction\/}  of a
copermutad $C$  is the differential graded (dg) 
permutad ${\cobar} (C) := (\PP(\desusp C),\pa_\cobar)$. The {\em dual bar
  construction\/} of a permutad $A$ satisfying the assumptions of
Example~\ref{Pisu_v_Koline_po_skoleni_ze_stavby.} is the dg permutad
$\D(A) := \cobar(A^*)$.
\end{definition}

\noindent
{\bf Convention.} From now on we will assume 
that $A$ is a binary quadratic permutad as in
Definition~\ref{Zavola_Petr?}, with $B(1)$ finite
dimensional.

Under the above assumption, one is allowed to define the Koszul dual {\em
  copermutad\/} by
\hbox{$A^\antishriek: = {A^!}^*$}, the
component-wise linear dual. With this notation, $\D(A^!) =
\cobar(A^\antishriek)$. 
As emphasized in \cite{loday-vallette}, 
the copermutad $A^\antishriek$ is more fundamental than $A^!$. It can be
defined directly, without the passage through
the permutad $A^!$, using coideals in cofree copermutads, without any
assumptions on the size of the generators of $A$. Given the
applications we had in mind, we however decided not to use this general
definition and keep working with more intuitive ideals and free permutads.

\begin{proposition}
\label{Pozvani_do_MSRI}
The dual bar construction $\D(\term^!)$ of\/ $\term^!$ is isomorphic to
the dg permutad $\CC_*(P)$ of cellular chains of the permutohedron.
\end{proposition}

\begin{proof}
The proof relies on the description of the permutad $\term^!$ given in
Proposition~\ref{Dnes_vecer_prijde_Jarka.}. 
Denote by \hbox{$e_n \in \, \downarrow \!\term^!(\underline n)^*$}, $n \!\geq\!
1$, the generator dual to  $\uparrow^{n-1} \hskip -.5em 
1 \in\, \uparrow^{n-1}\!
\bbk \cong \, \downarrow \!
\term^!(n) $. Then $e_1,e_2,\ldots$ are 
the free permutadic generators of $\D(\term^!)$. 
Using the description of the structure operations of $\term^!$ given
in Proposition~\ref{Dnes_vecer_prijde_Jarka.}, one 
easily verifies
that the differential
$\pa_{\, \D}$ in the dual bar construction  $\D(\term^!)$ is
\begin{equation}
\label{Dnes_s_Jarkou_a_Michaelem_na_obed.}
\pa_\D(e_n) = \sum_{r \in \Surj(\underline n,\underline 2)} 
 (-1)^{r_1 -1}\varepsilon(r)\cdot
(e_{r_1} \pcirc^\PP_r e_{r_1} ),\ n \geq 1,
\end{equation}
where  $r_1, r_2$ and the sign $\varepsilon(r)$ 
are as in~\eqref{Za_chvili_volam_Jarce.}, and  $\pcirc^\PP_r$ 
is the structure operation of the free permutad $\PP(\term^!)$.

Referring to the description of $\CC_*(P)$ given
in Example~\ref{Za_tyden_jedu_na_Zimni_skolu.}, we define an
isomorphism $\xi: \D(\term^!) \stackrel\cong\longrightarrow \CC_*(P)$ of free
permutads by \[
\xi(e_n) := -c_{n-1}, \ n \geq 1,
\] 
with $c_{n-1}$ the generator
in~\eqref{Je_streda_a_melu_z_posledniho.}. Comparing~\eqref{Dnes_s_Jarkou_a_Michaelem_na_obed.}
to the formula for the differentials of $c_n$'s 
in~\cite[\S9.3]{loday11:_permut} we realize that they match up to an
overall minus sign. Therefore $\xi$ defined above 
commutes with the differentials, 
thus constitutes the required isomorphism of dg permutads. 
\end{proof}

Let us proceed towards the Koszulity property of binary quadratic
permutads in Definition~\ref{Zavola_Petr?}.
One starts from a monomorphism $\susp B \hookrightarrow  A^!$ of
collections defined as
the composition
\[
\susp B \hookrightarrow \PP(\susp B) 
\twoheadrightarrow \PP(\susp B)/(S^\perp) = A^!.
\]
Its linear dual  is a surjection ${A^!}^* \twoheadrightarrow\ \susp
B$  which desuspens to 
a map $\pi:\ \desusp {A^!}^*
\twoheadrightarrow B$. The related {\em twisting
morphism\/}\footnote{We borrowed this terminology from
\cite[Chapter~6]{loday-vallette}.}  $\desusp {A^!}^*  \to A$ is 
the composition
\[
 \desusp  {A^!}^*\stackrel\pi\twoheadrightarrow  B 
 \hookrightarrow \PP(B) 
\twoheadrightarrow \PP(B)/(S) = A.
\]
It extends to a permutad morphism
$\rho : \PP(\desusp {A^!}^*) \to A$ by the freeness of $\PP(\desusp
{A^!}^*)$. One  easily verifies:

\begin{proposition}
One has
$\rho \circ \pa_{\,\D} = 0$, therefore $\rho$ induces the {\em
  canonical map\/} 
\begin{equation}
\label{Jarce_prijede_M1.}
\can:
\D(A^!) = (\PP(\desusp {A^!}^*), \pa_\D)    \longrightarrow (A,0)
\end{equation}  
of dg permutads.
\end{proposition}

\begin{definition}
A binary quadratic permutad $A$ is {\em Koszul\/} if the 
canonical map~\eqref{Jarce_prijede_M1.} is a~component-wise homology isomorphism.
\end{definition}

If $A$ is concentrated in degree $0$, 
clearly $H^0(\D(A^!)) \cong A$, thus such an $A$ is Koszul if and only if
$\D(A^!)$ is acyclic in positive dimensions. This observation will be
used in our proof of

\begin{theorem}
\label{Pan_Hladky_zase_del8_problemy.}
The terminal permutad $\term$ is Koszul.
\end{theorem}

\begin{proof}
The statement 
follows from the isomorphism $\D(\term^!) \cong \CC_*(P)$ established
in Proposition~\ref{Pozvani_do_MSRI} combined with the contractibility of the
permutohedron which implies that  $\CC_*(P)$ is acyclic in positive dimensions.
\end{proof}

Let us close this section by
a  Koszulity test in the spirit of the
Ginzburg-Kapranov criterion for 
operads~\cite[Theorem~3.3.2]{ginzburg-kapranov:DMJ94}. For 
$A = \coll An \in \Coll$ with finite-dimensional pieces we define its
generating power series as 
\[
f_A (t) := \sum_{n =1}^\infty \frac {\chi(A(\underline n))}{n!},
\]
in which $\chi$ denotes the Euler characteristic.
One then has

\begin{proposition}
Assume that $A$ is a binary quadratic permutad as in
Definition~\ref{Zavola_Petr?} with $B(1)$ finite dimensional. If $A$
is Koszul, 
then its generating series $f_A$ 
and the series $f_{A^!}$ of its Koszul dual are related by the
functional equation
\begin{equation}
\label{chripka_pred_odjezdem_do_Srni}
f_A(t) = \frac  {-f_{A^!}(t)}{1 + f_{A^!}(t)} \ .
\end{equation}
\end{proposition}

\begin{proof}
Suppose that $M = \coll Mn$ is a collection of graded vector spaces
with finite-dimensional pieces. Simple combinatorics gives
\[
\eul{\PP(M)(\underline n)} = \sum_{s \geq 1}\sum_{k_1 + \cdots + k_s =
  n}
\frac {n!}{k_1! \cdots k_s!} \eul{M(\underline k_{\, 1}}) \cdots
\eul{M(\underline k_{\,  s})}  
\]
which is the same as
\[
\frac{\eul{\PP(M)(\underline n)}}{n!} 
= \sum_{s \geq 1}\sum_{k_1 + \cdots + k_s =
  n}
\frac{\eul{M(\underline k_{\, 1})}}{k_1!} \cdots
\frac{\eul{M(\underline k_{\,  s})} }{k_s!}.
\]
Therefore the generating series of $M$ and that of the free permutad $\PP(M)$ are
related by
\[
f_{\PP(M)} = f_M + f_M^2 + f_M^3 + \cdots ={f_M}(1-f_M)^{-1}.
\]
When $A$ is Koszul, 
one has
$f_{\D(A^!)} = f_A $, while we already know from the above calculations that
\[
f_{\D(A^!)} = f_{P(\desusp  \hskip .15em A^!)} =  {f_{\desusp
\hskip .15em    A^!}}(1-f_{\desusp \hskip .15em A^!})^{-1} =   {-f_{
    A^!}}(1+f_{ A^!})^{-1}. 
 \]
This finishes the proof.
\end{proof}

\begin{example}
One has
\[
\textstyle
f_{\ssterm}(t) = t + \frac1{2!} t^2 + \frac1{3!} t^3 + \cdots = e^t-1
\]
while
\[
\textstyle
 f_{\ssterm^!}(t) = -t + \frac1{2!} t^2 - \frac1{3!} t^3 + \cdots = e^{-t}-1.
\]
Plugging it into~\eqref{chripka_pred_odjezdem_do_Srni} results in
\[
e^t-1 = \frac{ -(e^{-t}-1)}{1+ (e^{-t}-1)} = 
\frac{1 -e^{-t}}{e^{-t}} =  e^t-1
\]
as expected.
\end{example}

\begin{example}
\label{Budu_mit_v_Srni_samoztatny_pokoj.}
A `twisted' version of the terminal permutad $\term$ of
Example~\ref{Zitra_sraz_s_Jitkou_v_ustavu.} is the quotient
\[
\tildeterm : = \PP(\mu)/(\widetilde S)
\]
where 
$\mu \in
\PP(\mu)(\underline 1)$ is a degree $0$ generator and
$(\widetilde S)$ the permutadic ideal generated by
\[
\pcirc_{\id}(\mu \ot \mu) + \pcirc_{(21)}(\mu \ot \mu).
\]
Notice that $\tildeterm$ equals the `parametrized associative
permutad' \hbox{$q$-permAs} of \cite[\S4.5]{loday11:_permut} taken with $q=-1$.
Since $\tildeterm$ is a permutadic version of the operad for antiassociative
algebras which serves as 
a standard example of a non-Koszul operad~\cite[Section~5]{galgal},
one would expect that  $\tildeterm$ is non-Koszul as well. 

Surprisingly, it is not so. It turns out that that the dg collections 
$\D(\tw)$ and $\D(\term^!)$ are isomorphic, though they
are \underline{not} isomorphic as dg permutads. This however suffices
for the acyclicity of  $\D(\tw)$ in positive
dimensions, and thus the Koszulity of $\tildeterm$.

The isomorphism 
$\zeta:\D(\tw) \stackrel\cong\longrightarrow  \D(\term^!)$ 
of dg collections is constructed as follows.
As in the proof of Proposition~\ref{Dnes_vecer_prijde_Jarka.}
we establish that
\[
\tw(\underline n) \cong\ \susp^{\hskip .15em n}\,\bfk,\ n \geq 1,
\] 
with the structure operations~\eqref{Pozitri_jedem_do_Zvanovic.}
the canonical isomorphisms 
\[
\hbox{$\susp^{\hskip .15em r_1} \hskip .1em \bbk$} \ot \hbox{$\susp^{\hskip .15em r_2} 
\hskip .1em \bbk$}\hskip .1em \cong\ \hbox{$\susp^{\hskip .15em n}\,\bbk$}
\]
without any additional sign factor. The calculation is actually
even simpler than in the case of the untwisted $\term$, since the relation
induced by the ideal $(\widetilde S^\perp)$ does not involve any
signs. We infer that $\D(\tw)$ is freely generated by degree
$n\-1$ generators $\tilde e_n$, $n \geq 1$, whose differential is given
by a formula analogous to~\eqref{Dnes_s_Jarkou_a_Michaelem_na_obed.}
but without the $\varepsilon(r)$-factor.
The underlying permutad of the dg permutad $\D(\tw)$ is then described as 
\[
\PP\big(\desusp {\tw}^*\big) (\underline n)\cong
 \bigoplus_{1 \leq k \leq n} \  \bigoplus_{r \in
  \Surj(\underline n,\underline k)}
\Span(\tilde e_{r_1} \ot \cdots  \ot \tilde e_{r_k}), \ n \geq 1,
\]
where $r_i$ is the cardinality of the set-theoretic preimage of $i \in
\underline k$ via the map $r : \underline n \epi \underline k$\,. We
have a similar formula for the underlying permutad of $\D(\term)$, namely
\[
\PP\big(\desusp {\term^!}^*\big)  (\underline n) \cong
 \bigoplus_{1 \leq k \leq n} \  \bigoplus_{r \in
  \Surj(\underline n,\underline k)}
\Span(e_{r_1} \ot \cdots  \ot e_{r_k}),\ n \geq 1.
\]
The isomorphism 
$\zeta:\D(\tw) \stackrel\cong\longrightarrow  \D(\term^!)$ is, under
the above identifications, given by
\[
\zeta(\tilde e_{r_1} \ot \cdots  \ot \tilde e_{r_k}) : =
\varepsilon(r)
\cdot (e_{r_1} \ot \cdots  \ot e_{r_k}),
\]
with $\varepsilon(r)$ the sign of the shuffle associated to $r$.
It is simple to verify
that $\zeta$ defined this way commutes with the differentials of the
dual bar constructions, so it is an isomorphism of dg collections. 
On the other hand,
$\zeta$ is clearly not a morphism of permutads. 
\end{example}

\part{Operadic category approach}

\section{Permutads as algebras over the terminal operad}

The fundamental feature of Batanin-Markl's (BM) theory of operadic
categories~\cite{duodel} is that the objects under
study are viewed as algebras over (generalized) operads in a specific
operadic category, cf.\ also the introduction to \cite{Sydney}.
Thus, for instance, ordinary operads are algebras over the terminal
operad ${\sf 1}_\RTr$ in the operadic category $\RTr$ of rooted
trees. 
The operad ${\sf 1}_\RTr$ is
quadratic self-dual which, according to BM theory, implies that the bar
constructions of its algebras (i.e.~operads) are algebras of the same type
(i.e.~operads) again. As we noticed in the introduction, the same is
true for permutads.

Let us start by recalling,
following~\cite[\S14.4]{Sydney}, the operadic category $\ttP$ that plays for
permutads the same r\^ole as $\RTr$ for ordinary operads.\footnote{We
  refer to \cite[\S I.1]{duodel}  again for the language of operadic categories.}
Objects
of $\ttP$ are surjections 
\hbox{$\alpha: \underline n \epi \underline k \in \Surj(n,k)$},
$n \! \geq \! k \geq \! 1$, and morphisms $f : \alpha' \to \alpha''$ 
of $\ttP$ are~diagrams
\begin{equation}
\label{nikdo_mi_nepopral_k_narozeninam}
\xymatrix@C=2.5em{\underline n \ar@{=}[r]
\ar@{->>}[d]_{\alpha'} & \underline n \ar@{->>}[d]^{\alpha''}
\\
\underline k' \ar[r]^\gamma & \underline k''
}
\end{equation}
in which $\gamma$ is order-preserving (and necessarily a
surjection). 

The cardinality functor is defined by $|\alpha:
\underline n \epi \underline k| := k$. The $i$-th fiber $\inv f(i)$ of the
morphism~\eqref{nikdo_mi_nepopral_k_narozeninam} is the surjection
$(\gamma\alpha')^{-1}(i) \epi \gamma^{-1}(i)$, $i \in \underline
k''$. The only local terminal objects of $\ttP$ are the surjections 
\begin{equation}
\label{Zprava_na_grantu.}
U_n :=
\underline n \epi \underline
1,\ n \geq 1,
\end{equation} 
which are also the chosen  (trivial) ones. 
All \qb{s}, and isomorphisms in general, are the identities.

Let us  give also the following
explicit description of the fiber  $\inv f(i)$ of
$f$ in~\eqref{nikdo_mi_nepopral_k_narozeninam}.
Assume that
$k'_1 < \cdots < k'_t$ are numbers such that
\[
\{\Rada{k'}1t\} = \{j \in \underline k' \ | \ \gamma(j) = i \}
\]
and $n_1 < \cdots < n_s$ are such that
\[
\{\Rada{n}1s\} = \{u \in \underline n \ | \ \gamma\alpha'(u) = i \}.
\]
Then $\inv f(i)\in \ttP$ is the surjection $\underline s \epi \underline t$ 
that sends $u \in \underline s$ into the unique $j \in \underline t$ such
that $\alpha'(n_u) = k'_j$.

A $\ttP$-operad $\oP$ is a collection of vector
spaces $\{\oP(\alpha)\}_{\alpha \in \ttP}$ equipped with the structure
operations
\begin{subequations}
\begin{equation}
\label{dva_dny_za_sebou}
m_f : \oP(\alpha'') \ot \oP(f_1) \ot \cdots \ot \oP(f_{k''}) 
\longrightarrow  \oP(\alpha')
\end{equation}
specified for any
morphism $f : \alpha' \to \alpha''$ in $\ttP$ with fibers $\Rada
f1{k''}$, and the unit morphisms 
\begin{equation}
\label{Dvakrat za sebou jsem spal s Jarkou.}
\eta_n  : \bfk \to \oP(U_n), \ n \geq 1,
\end{equation}
\end{subequations}
subject to the axioms listed in \cite[Def.~1.11]{duodel}.

A particular example is the {\em
  terminal $\ttP$-operad\/} $\Pterm$ with \hbox{$\Pterm(\alpha):= \bfk$} for
each $\alpha\in \ttP$, the operations~\eqref{dva_dny_za_sebou} 
the natural identification
$\bfk \ot \bfk^{\otimes k''} \cong \bfk$ and $\eta_n$
in~\eqref{Dvakrat za sebou jsem spal s Jarkou.} the identity $\bfk=\bfk$.\footnote{We are again slightly abusing the
  terminology; $\Pterm$ is in fact the linearization of the terminal
  $\ttP$-operad in the Cartesian monoidal category of sets.}    
Its algebras are described in

\begin{proposition}
\label{Treti_den_je_kriticky.}
Algebras for the terminal $\ttP$-operad $\Pterm$, in the sense
of Definition~1.20 of~\cite{duodel}, are permutads. 
\end{proposition}

\begin{proof}
The statement is a part of~\cite[Theorem~14.4]{Sydney} whose proof
uses an explicit  presentation of
$\Pterm$, but we will show directly that  $\Pterm$-algebras
are the same as algebras for the monad~$\PP$ in
Definition~\ref{Posledni_moznost_jet_na_kole_a_ja_mam_chripku.}. 
Let $f : \alpha' \to \alpha''$, $\alpha' \in \Surj(n,k')$, $\alpha''
\in \Surj(n,k'')$, be a morphism in $\ttP$ with fibers $\Rada
f1{k''}$. The crucial fact on which
our proof is based is that
\begin{equation}
\label{Asi_dneska_ten_Kolin_pustim.}
\alpha' = (\alpha''; \Rada f1{k''}),
\end{equation}
where $(-;\rada --)$ is the substitution introduced in
\cite[Section~1.2]{loday11:_permut}. Formula~\eqref{Asi_dneska_ten_Kolin_pustim.}
follows directly from definitions, as the reader may check easily.

Let us inspect what~\cite[Definition~1.20]{duodel} gives in our
case. The underlying spaces of an algebra over an operad in
an operadic category $\ttO$ is indexed by the set $\pi_0(\ttO)$ of its
connected components, which is isomorphic to the set of chosen
terminal objects. We identify
\[
\pi_0(\ttP) = \{\underline 1,
\underline 2, \underline 3, \ldots\},
\]
thus $\Pterm$-algebras are collections $A = \{A(\underline n) \in \Vect\ | \
\underline n \in \Fin\} \in \Coll$. By~\cite[Definition~1.20]{duodel}
again, the structure maps of an algebra $A$ over an operad $\oP$ in
an operadic category $\ttO$ are morphisms
\begin{equation}
\label{Mozna_nakonec_do_Kolina_pojedu.}
\mu_T : \oP(T) \longrightarrow   
\Lin\big(\bigotimes_{c \in \pi_0(s(T))} A(c),
A(\pi_0(T))\big),
\end{equation}
where $T \in \ttO$, $\pi_0(s(T))$ is the subset of $\pi_0(\ttO)$
formed by
the connected components of $\ttO$ to which the
fibers of the identity map $\id : T \to T$ belong, and
$\pi_0(T)$ is the connected component of~$T$. In our case, $\oP$ is
the constant $\ttP$-operad $\Pterm$ whose each piece  equals $\bbk$. 
If $\alpha : \underline n \epi \underline k \in
\ttP$,~clearly 
\[
\pi_0(s(\alpha)) = \{\underline n_1,\ldots, \underline n_k\},\ \underline
n_i := \alpha^{-1} (i), \ 1 \leq i \leq k,
\]
while $\pi_0(\alpha) = \underline n$. The structure maps associated to such an
$\alpha$ are therefore given by
\begin{equation}
\label{Dnes_byla_vojenska_prehlidka.}
m_\alpha := \mu_\alpha(1) : A(\underline n_1) \ot \cdots \ot A(\underline n_k)
\longrightarrow  A(\underline n).
\end{equation}
According to~\cite[Definition~1.20]{duodel}, the structure
maps~\eqref{Mozna_nakonec_do_Kolina_pojedu.} must assemble to 
a morphism $\mu : \oP \to \End^{\,\ttO}_A$ of
$\oP$ to the endomorphism operad of the collection $A$. 
Let us inspect what it means in our case.

First of all, $\mu : \Pterm \to \End^{\,\ttP}_A$ must preserve operad
units. This means that for each $\alpha$ which is chosen local 
terminal, i.e.\
$\alpha  : \underline n \epi \underline 1$ for some $n \geq 1$, 
the diagram  
\[
\xymatrix{\Pterm(\alpha) \ar@{=}[dr]\ar[rr]^{\mu_{\alpha}}   & & \End^{\,\ttP}_A(\alpha)
\\
&\ar[ur]^(.4){\eta_\alpha}  \ \bbk,&
}
\]
in which $\eta_\alpha$ is the unit morphism of the endomorphism
operad, commutes. This is the same as to require that, 
for $\alpha: \underline n \epi \underline 1$, the structure map  
\[
m_\alpha : A(\underline n)
\longrightarrow  A(\underline n)
\]
equals the identity. It thus bears no information, so we consider
$m_\alpha$'s in~\eqref{Dnes_byla_vojenska_prehlidka.} only for
$|\alpha| \geq 2$.

Next, we must verify that $\mu : \Pterm \to \End^{\,\ttP}_A$ commutes
with the operadic structure operations.
Assume therefore that $f : \alpha' \to \alpha''$ is a morphism
in~\eqref{nikdo_mi_nepopral_k_narozeninam}, with fibers $\Rada f1{k''}$.
Its $i$th fiber $f_i$, $1 \leq i \leq  {k''}$, belongs to 
$\Surj(n_i,l_i)$, where $\underline l_i := \gamma^{-1}(i)$ and
$n_i\geq 1$ are such that 
\[
n_1 + \cdots + n_{k''} = n.
\]
Moreover,  
\[
\pi_0(s(f_i)) = \{\underline n^1_i,\ldots,\underline n^{l_i}_i\},  \ 1
\leq i \leq k'', 
\]
with some $n^1_i,\ldots,n^{l_i}_i\geq 1$ such that
\[
n^1_i + \cdots + n^{l_i}_i = n_i, \ 1 \leq i \leq k'',
\]
while $\pi_0(f_i) = \underline n_i$. Notice also that
\[
\pi_0(s(\alpha')) = \{\underline n^1_1,\ldots,\underline n^{l_1}_1,
\ldots, \underline n^1_{k''},\ldots,\underline n^{l_{k''}}_{k''}\}
\]
and that
\[
\pi_0(s(\alpha'')) = \{\rada{\underline n_1}{\underline n_{k''}}\}.
\]
One easily finds that  $\mu : \Pterm \to \End^{\,\ttP}_A$ commutes
with the structure operations 
of $\ttP$-operads if and only if, for each $f$ as above, the diagram
\[
\xymatrix{\displaystyle\bigotimes_{1 \leq i \leq k''}
  \Pterm(f_i) \ot  \Pterm(\alpha'') 
\ar[r]^(.6){\gamma_f}_(.6)\cong 
\ar[d]_{\bigotimes_{i} \mu_{f_i} \ot  \mu_{\alpha''}}
& \Pterm(\alpha') \ar[d]^{\mu_{\alpha'}}
\\
\ar[r]^(.659){\comp}
\displaystyle
\bigotimes_{1 \leq i \leq k''} \Lin\big(
\bigotimes_{1 \leq s \leq l_i}   A(\underline n^s_i),A(\underline
n_i)\big)
\ot\Lin\big( \bigotimes_{1 \leq i \leq k''} A(\underline n_i),A(\underline n)\big)
&
\displaystyle
\Lin\big(\hskip -.3em 
\bigotimes_\doubless{1 \leq i \leq k''}{1 \leq s \leq l_i}   
A(\underline n^s_i),A(\underline
n)\big)
}
\]
in which  $\gamma_f$ is the operatic composition in $\Pterm$ and
${\comp}$ the obvious composition of linear maps, commutes. Since, in our case,
\[
 \bigotimes_{1 \leq i \leq k''}
  \Pterm(f_i)  \ot \Pterm(\alpha'')\cong \bbk \cong  \Pterm(\alpha')
\] 
and $\gamma_f$ is, under this identification, the identity, 
the commutativity of the above diagram
is equivalent to the equation
\[
{\comp}(\bigotimes_{1 \leq i \leq
    k''} m_{f_i} \ot m_{\alpha''}) = m_{\alpha'}
\]
involving the structure maps~\eqref{Dnes_byla_vojenska_prehlidka.}.
In other words, one requires that,
whenever~\eqref{Asi_dneska_ten_Kolin_pustim.} is satisfied, 
\[
m_{\alpha'} = m_{\alpha''}(m_{f_1},\ldots,m_{f_{k''}})
\]
which is an equality of maps
\[
A(\underline n^1_1) \ot \cdots \ot A(\underline n^{l_1}_1) 
\ot \cdots \ot A(\underline n^1_{k''}) \ot \cdots \ot A(\underline
n^{l_{k''}}_{k''}) \longrightarrow A(\underline n).
\]

 To finish
the proof, it is
enough to realize how the
substitution~\eqref{Asi_dneska_ten_Kolin_pustim.} enters, in the proof
of \cite[Proposition~1.4]{loday11:_permut},  the
definition of the map $\Gamma_E$ in~\eqref{Pujdu_si_nakoupit.}. The
fact 
that $\PP$-algebras are indeed the same as the structures described
above will then be self-evident.
\end{proof}

\section{Free $\ttP$-operads}

This section is a preparation  for the construction of the
minimal model of the terminal $\ttP$-operad given in
Section~\ref{Pujdeme_dnes_s_Jarcou_na_muslicky?} and for the
introduction of Koszul duality and Koszulity of $\ttP$-operads given
in Sections~\ref{Slepim_si_i_410?} and~\ref{Zase_mam_chripku!}.  
Free operads were, in the context of
general operadic categories, addressed in~\cite[Section~10]{Sydney} to
which we refer for the terminology used in the following
sentences. Our situation is
however simplified by the fact that all local terminal objects in
$\ttP$ are the chosen terminal ones, thus unital $\ttP$-operads are
automatically strictly extended unital. 
Moreover, for our purposes it
suffices to consider only $1$-connected $\ttP$-operads, i.e.\ to operads
$\oP$ such that $\oP(\alpha) \cong \bbk$ if $|\alpha| = 1$.

With the above in mind, we introduce the category $\Collect$  
whose objects are collections 
$E = \{E(\alpha) \ | \ \alpha \in \ttP\}$ of graded vector spaces with
$E(\alpha)= 0$ if $|\alpha| = 1$, and level-wise morphisms.
One has the obvious forgetful functor
\[
\Box : \Alg \longrightarrow \Collect
\]   
from the category of $1$-connected unital $\ttP$-operads
given by
\[
\Box\, \oP(\alpha) := 
\begin{cases}
  \oP(\alpha) &\hbox {if $|\alpha| \geq 2$, and}
\\
0& \hbox{otherwise.} 
\end{cases}
\]
Theorem~10.11 of~\cite{Sydney} guarantees that $\Box$ admits a left
adjoint $\F: \Collect \to \Alg$. We call $\F(E)$ for $E\in \Collect$
the {\em free $\ttP$-operad\/} generated by a $1$-connected 
$\ttP$-collection~$E$.

We describe in detail $\F(E)$ for generating collections of a particular
form. The central r\^oles in this simplified construction
will be played by the classical
free non-$\Sigma$ operads. This generality 
is sufficient for all concrete applications
given in the rest of this paper. The construction of $\F(E)$ for an arbitrary 
$1$-connected $\ttP$-collection $E$ is sketched out at the end of this section.

Let $\Ord$ be the subcategory of $\Fin$ with the same objects and
order-preserving surjections as morphisms. It is an operadic category
whose operads are the classical constant-free
\hbox{non-$\Sigma$} (non-symmetric)
operads~\cite[Example~1.15]{duodel}. 
One has a strict operadic functor, in the sense 
of~\cite[p.~1635]{duodel}, 
$\des: \ttP \to \Ord$ given on objects by
$\des(\underline n \epi \underline k) := \underline k$ while, for a morphism $f$
in~\eqref{nikdo_mi_nepopral_k_narozeninam}, one puts $\des(f) := \gamma$.
According to general theory~\cite[p.~1639]{duodel}, each $\Ord$-operad $\uoS$ 
determines,\footnote{We use the convention pioneered in
  \cite{markl-shnider-stasheff:book} and distinguish  classical non-$\Sigma$
  operads by \underline{inderl}y\underline{nin}g.}   
via the restriction along $\des$, a $\ttP$-operad
$\des^*(\uoS)$ such that $\des^*(\uoS)(\alpha) = \uoS(\des(\alpha))$,
$\alpha \in \ttP$.

The form of the structure map $m_f$ of a $\ttP$-operad $\oP$
associated to a morphism $f : \alpha' \to \alpha''$ in $\ttP$ with
fibers $\Rada f1{k''}$ is recalled in~\eqref{dva_dny_za_sebou}.  For
$\oP = \des^*(\uoS)$, $m_f$ is determined by the structure operation
$\underline m_\gamma: \uoS(k'') \ot \uoS(l_1) \ot \cdots \ot
\uoS(l_{k''}) \to \uoS(k')$ of $\uoS$ associated 
to $\gamma = \des(f)$\footnote{The form of structure operations
  of non-$\Sigma$ operads is recalled 
in~\eqref{Dostal jsem obrovskou lahev rumu.} below.}  via the
commutativity of the diagram
\[
\xymatrix{
\des^*(\uoS)(\alpha'') \ot \des^*(\uoS)(f_1) \ot \cdots \ot
\des^*(\uoS)(f_{k''}) \ar[d]_{m_f} \ar@{=}[r]
&
\uoS(k'') \ot \uoS(l_1) \ot \cdots \ot \uoS(l_{k''}) \ar[d]^{\underline m_\gamma}
\\
\des^*(\uoS)(\alpha') 
 \ar@{=}[r]
& \ \uoS(k'). 
}
\]
Let $E = \{E(\alpha)\}_{\alpha \in \ttP}$ be a collection of vector
spaces such that 
\begin{equation}
\label{Dnes_jsem_daval_dalsi_rozhovor.}
E(\alpha' : \underline n' \epi \underline k') =  
E(\alpha'' : \underline n'' \epi \underline k'') \ \hbox { if } \ k' = k''  
\end{equation}
and $\underline E = \coll{\underline E}k \in \Coll$ be defined as
\begin{equation}
\label{Pujdu_si_koupit_ztracene_pouzdro_na_bryle.}
\underline E(\underline k) := E(\alpha : \underline n \epi \underline k)
\end{equation} 
for an arbitrary $\alpha : \underline n \epi \underline k$, $k \geq 1$.
Let finally $\uF(\underline E)$ be the free
non-$\Sigma$ operad generated by a collection $\underline E$ 
above, and $\F(E) := \des^* (\uF(\underline E))$.

\begin{proposition}
\label{Leopold}
The  $\ttP$-operad $\F(E) = \des^* (\uF(\underline E))$ is the free  $\ttP$-operad generated by
the collection $E$. It is naturally graded,
\[
\F(E)(\alpha) = \bigoplus_{s \geq 0} \F^s(E)(\alpha),\ \alpha \in \ttP.
\]
The grading is such that 
$\F^1(E)(\alpha) = E(\alpha)$ for all $\alpha \in \ttP$, and
\begin{equation}
\label{Zitra_sraz_s_Nimou.}
\F^0(E)(\alpha: \underline n \epi \underline k) =
\begin{cases}
\bbk & \hbox {if $k=1$, and}
\\
0 & \hbox {otherwise.}  
\end{cases}
\end{equation}
\end{proposition}

\begin{proof}
  The claim is immediately obvious from the explicit description of free
operads given in~\cite[Section~10]{Sydney}, but we give an independent proof.
Recall from~\cite[\S5.2]{Sydney} 
that a strict operadic functor $p:\ttO\to {\tt P}$ is  a {\em discrete
    opfibration\/} if 
\begin{itemize}
\item[(i)]
$p$ induces a surjection $\pi_0(\ttO) \twoheadrightarrow
\pi_0(\tt P)$ and
\item[(ii)]
for any morphism $f : T\to S$ in $\tt P$ and any
  $t \in \ttO$ such that
  $  p(t) = T$
  there exists a~unique $\sigma : t\to s$ in $\ttO$ such that
  $p(\sigma) = f$.
\end{itemize}
By dualizing~\cite[Theorem~2.4]{duodel} one verifies that the
restriction  $p^* :\Op{\tt P} \to \Op{\ttO}$ between the
associated categories of operads has a right
adjoint $p_* : \Op{\ttO} \to \Op{\tt P}$ defined on objects by
\[
p_*(\oO)(T) := \prod_{p(t) = T} \oO(t), \ \oO \in  \Op{\ttO}, \
T \in \tt P.
\]
One thus has the adjunction
\begin{equation}
\label{Zitra_snad_naposledy_predsedam_komisi.}
\Op{\ttO}\big(p^*(\oS),\oO\big) \cong 
\Op{\tt P}\big(\oS,p_*(\oO)\big),\
\oS \in \Op{\tt P},\ \oO \in \Op{\tt P}.
\end{equation}

It is easy to check that $\des : \ttP \to \Ord$ is a discrete
opfibration, therefore the
adjunction~\eqref{Zitra_snad_naposledy_predsedam_komisi.}, with
$\uF(\underline E)$ in place of $\oS$ and a non-$\Sigma$ constant-free
operad $\uoO$ in place of~$\oO$,~gives
\[
\Op{\ttP}\big(\F(E),\uoO\big) \cong 
\Op{\Ord}\big(\uF(\underline E),\des_*(\uoO)\big).
\]
Invoking the fact that $\uF(\underline E)$ is the free
non-$\Sigma$-operad, one sees that
the right hand side of the above isomorphism consists of families of
linear maps
\begin{equation}
\label{Hilzner}
\{\underline E(\underline k) \to \des_*(\uoO)(\underline k)\ | \
k \geq 1\}.
\end{equation}
Since, by definition, 
$\des_*(\uoO)(\underline k) = \prod_{\alpha
: \underline n \to \underline k} \uoO(\alpha)$, taking into account the
definition~\eqref{Pujdu_si_koupit_ztracene_pouzdro_na_bryle.} of $E$,
one sees that the family in~\eqref{Hilzner} is the same as a family of
linear maps
\begin{equation}
\label{Hruzova}
\{ E(\alpha) \to \uoO(\alpha) \ | \
\alpha \in  \ttP\}.
\end{equation}
In other words, $\ttP$-operad maps $\F(E) \to \uoO$ are in a natural
one-to-one correspondence with families~\eqref{Hruzova}.  This makes
the freeness of $\F(E)$ obvious.

The free non-$\Sigma$ operad  $\uF(\underline E)$ is naturally graded,
with  $\uF^1(\underline E) = \underline
E$~\cite[page~1475]{markl:zebrulka}, and this grading clearly
induces a grading of  $\F(E) = \des^* (\uF(\underline E))$ having the
requisite properties.
This finishes the proof.
\end{proof}

Let us point out that~\eqref{Dnes_jsem_daval_dalsi_rozhovor.}
characterizes collections $E \in \Collect$ that are the restrictions  of
some collection 
$\underline E \in \Collectord$
along $\des: \ttP \to \Ord$, i.e.\ that are of the form
$E = \des^*(\underline E)$ for some $\underline E$.

\begin{remark}
As argued in~\cite[Subsect.~5.2]{Sydney}, each discrete operadic opfibration
$p : \ttO \to \tt P$ 
is obtained from a certain $\tt P$-cooperad $\oC$ via the operadic
Grothendieck construction. In our concrete case the corresponding 
$\Ord$-cooperad $\oC$ is given by
\[
\oC(\underline k) := \coprod_{n \geq k}\Surj(\underline n,\underline
k), \ k \geq 1.
\]
\end{remark}

\noindent
{\bf General case.}
Let $E$ be an arbitrary $1$-connected 
$\ttP$-collection. The constructions
in Section~10 of~\cite{Sydney} 
specialize to the  description of the free $\ttP$-operad $\F(E)$
generated by $E$ given below. We however need some notation.

For $\alpha : \underline n \epi \underline k\in \ttP$ and $s \geq 1$,
let $\Tr^s(\alpha)$ be the set of planar rooted trees growing from the
bottom up with $s$ at least binary
vertices, and leaves labeled from the left to right by the elements of
the ordered set $\underline k$. 
We extend this definition by postulating $\Tr^0(\alpha) := \emptyset$ if
$|\alpha|> 1$, while for $|\alpha| =1$,
$\Tr^0(\alpha)$  
consists of a singular rooted tree with one leaf and no~vertex. 

Each vertex $v \in \vert(\tau)$ of $\tau \in \Tr^s(\alpha)$ 
determines a segment $\underline k_v
\subset \underline k$ of those 
$i \in\underline k$ for which $v$ lies
on the path in $\tau$ connecting the leaf labeled by $i$ with
the root.
We then define~$\underline n_v$ to be the ordinal 
given by the pullback
\[
\xymatrix@C=3em{\underline n_v\ar[r]\ar[d] & \underline n \ar[d]^\alpha
\\
\underline k_v\ar@{^{(}->}[r] & \underline k.
}
\] 
Consider finally the surjection
\begin{equation}
\label{Dnes_v_Mnisku}
\xymatrix@1{\alpha_v:
\underline n_v \ar@{->>}[r]  &\ \In(v)
}
\end{equation}
to the set $\In(v)$ of edges incoming to $v$ which sends $j \in
\underline n_v$ to the unique edge in the path
connecting $\alpha(j)$ to the root of $\tau$. The set $\In(v)$ is
ordered by the clockwise orientation of the plane, and~\eqref{Dnes_v_Mnisku} is
order-preserving. We will thus interpret it as an object of  
$\ttP$, cf.~the remark after
diagram~\eqref{S_Jarkou_do_Mnisku.}. We believe that
Figure~\ref{Opet_jsem_podlehnul.} makes these definitions clear.
\begin{figure}
\[
\psscalebox{.8 .8} 
{
\begin{pspicture}(0,-4.51)(18.203558,4.51)
\psline[linecolor=black, linewidth=0.04](4.905,1.09)(14.505,1.09)
\psline[linecolor=black, linewidth=0.04](4.905,1.09)(9.705,-3.71)(14.505,1.09)
\psdots[linecolor=black, dotsize=0.2](4.905,1.09)
\psdots[linecolor=black, dotsize=0.2](6.105,1.09)
\psdots[linecolor=black, dotsize=0.2](7.305,1.09)
\psdots[linecolor=black, dotsize=0.2](8.505,1.09)
\psdots[linecolor=black, dotsize=0.2](9.705,1.09)
\psdots[linecolor=black, dotsize=0.2](10.905,1.09)
\psdots[linecolor=black, dotsize=0.2](12.105,1.09)
\psdots[linecolor=black, dotsize=0.2](13.305,1.09)
\psdots[linecolor=black, dotsize=0.2](14.505,1.09)
\psline[linecolor=black, linewidth=0.04](9.705,1.09)(10.505,0.29)(10.905,1.09)
\psline[linecolor=black, linewidth=0.04](12.105,1.09)(12.105,-0.11)(13.305,1.09)(13.305,1.09)
\psline[linecolor=black, linewidth=0.04](6.105,1.09)(8.505,-0.91)(8.505,1.09)
\psline[linecolor=black, linewidth=0.04](7.305,1.09)(8.505,-0.91)(9.705,-3.71)(9.705,-3.71)
\psdots[linecolor=black, dotsize=0.2](2.505,3.89)
\psdots[linecolor=black, dotsize=0.2](3.705,3.89)
\psdots[linecolor=black, dotsize=0.2](4.905,3.89)
\psdots[linecolor=black, dotsize=0.2](6.105,3.89)
\psdots[linecolor=black, dotsize=0.2](7.305,3.89)
\psdots[linecolor=black, dotsize=0.2](8.505,3.89)
\psdots[linecolor=black, dotsize=0.2](9.705,3.89)
\psdots[linecolor=black, dotsize=0.2](10.905,3.89)
\psdots[linecolor=black, dotsize=0.2](12.105,3.89)
\psdots[linecolor=black, dotsize=0.2](13.305,3.89)
\psdots[linecolor=black, dotsize=0.2](14.505,3.89)
\psdots[linecolor=black, dotsize=0.2](15.705,3.89)
\psdots[linecolor=black, dotsize=0.2](16.905,3.89)
\psline[linecolor=black, linewidth=0.02, arrowsize=0.05291667cm 3.0,arrowlength=4.0,arrowinset=0.0]{->}(4.905,3.89)(9.705,1.09)
\psline[linecolor=black, linewidth=0.02, arrowsize=0.05291667cm 3.0,arrowlength=4.0,arrowinset=0.0]{->}(16.905,3.89)(9.705,1.09)
\psline[linecolor=black, linewidth=0.02, arrowsize=0.05291667cm 3.0,arrowlength=4.0,arrowinset=0.0]{->}(10.905,3.89)(10.905,1.09)
\psline[linecolor=black, linewidth=0.02, arrowsize=0.05291667cm 3.0,arrowlength=4.0,arrowinset=0.0]{->}(8.505,3.89)(10.905,1.09)
\psline[linecolor=black, linewidth=0.02, arrowsize=0.05291667cm 3.0,arrowlength=4.0,arrowinset=0.0]{->}(13.305,3.89)(12.105,1.09)
\psline[linecolor=black, linewidth=0.02, arrowsize=0.05291667cm 3.0,arrowlength=4.0,arrowinset=0.0]{->}(14.505,3.89)(13.305,1.09)
\psline[linecolor=black, linewidth=0.02, arrowsize=0.05291667cm 3.0,arrowlength=4.0,arrowinset=0.0]{->}(3.705,3.89)(12.105,1.09)
\psdots[linecolor=black, dotsize=0.2](18.105,3.89)
\psdots[linecolor=black, dotsize=0.2](1.305,3.89)
\psline[linecolor=black, linewidth=0.02, linestyle=dashed, dash=0.17638889cm 0.10583334cm, arrowsize=0.05291667cm 3.0,arrowlength=4.0,arrowinset=0.0]{->}(6.105,3.89)(4.905,1.09)
\psline[linecolor=black, linewidth=0.02, linestyle=dashed, dash=0.17638889cm 0.10583334cm, arrowsize=0.05291667cm 3.0,arrowlength=4.0,arrowinset=0.0]{->}(2.505,3.89)(7.305,1.09)
\psline[linecolor=black, linewidth=0.02, linestyle=dashed, dash=0.17638889cm 0.10583334cm, arrowsize=0.05291667cm 3.0,arrowlength=4.0,arrowinset=0.0]{->}(7.305,3.89)(6.105,1.09)
\psline[linecolor=black, linewidth=0.02, linestyle=dashed, dash=0.17638889cm 0.10583334cm, arrowsize=0.05291667cm 3.0,arrowlength=4.0,arrowinset=0.0]{->}(9.705,3.89)(8.505,1.09)
\psline[linecolor=black, linewidth=0.02, linestyle=dashed, dash=0.17638889cm 0.10583334cm, arrowsize=0.05291667cm 3.0,arrowlength=4.0,arrowinset=0.0]{->}(1.305,3.89)(14.505,1.09)
\psline[linecolor=black, linewidth=0.02, linestyle=dashed, dash=0.17638889cm 0.10583334cm, arrowsize=0.05291667cm 3.0,arrowlength=4.0,arrowinset=0.0]{->}(12.105,3.89)(7.305,1.09)
\psline[linecolor=black, linewidth=0.02, linestyle=dashed, dash=0.17638889cm 0.10583334cm, arrowsize=0.05291667cm 3.0,arrowlength=4.0,arrowinset=0.0]{->}(15.705,3.89)(6.105,1.09)
\psline[linecolor=black, linewidth=0.02, linestyle=dashed, dash=0.17638889cm 0.10583334cm, arrowsize=0.05291667cm 3.0,arrowlength=4.0,arrowinset=0.0]{->}(18.105,3.89)(14.505,1.09)
\psellipse[linecolor=black, linewidth=0.04, linestyle=dashed, dash=0.17638889cm 0.10583334cm, dimen=outer](11.505,1.09)(2.2,0.4)
\psdots[linecolor=black, dotsize=0.4](3.705,3.89)
\psdots[linecolor=black, dotsize=0.4](4.905,3.89)
\psdots[linecolor=black, dotsize=0.4](8.505,3.89)
\psdots[linecolor=black, dotsize=0.4](10.905,3.89)
\psdots[linecolor=black, dotsize=0.4](13.305,3.89)
\psdots[linecolor=black, dotsize=0.4](14.505,3.89)
\psdots[linecolor=black, dotsize=0.4](16.905,3.89)
\psline[linecolor=black, linewidth=0.04](10.505,0.29)(10.105,-2.51)(12.105,-0.11)
\psline[linecolor=black, linewidth=0.04](10.105,-2.51)(9.705,-3.71)
\psdots[linecolor=black, dotsize=0.2](10.105,-2.51)
\rput[b](3.705,4.29){$1$}
\rput[b](4.905,4.29){$2$}
\rput[b](8.505,4.29){$3$}
\rput[b](10.905,4.29){$4$}
\rput[b](13.305,4.29){$5$}
\rput[b](14.505,4.29){$6$}
\rput[b](16.905,4.29){$7$}
\rput[b](9.705,-2.51){\Large $v$}
\psline[linecolor=black, linewidth=0.1](9.705,-3.71)(9.705,-4.51)
\rput(10.505,-0.91){$1$}
\rput(11.305,-0.91){$2$}
\pscircle[linecolor=black, linewidth=0.04, fillstyle=solid, dimen=outer](10.3,-0.91){0.3}
\pscircle[linecolor=black, linewidth=0.04, fillstyle=solid, dimen=outer](11.45,-0.91){0.3}
\psdots[linecolor=black, dotstyle=o, dotsize=0.2, fillcolor=white](3.705,3.89)
\rput(10.29,-0.91){$1$}
\rput(11.45,-0.91){$2$}
\psdots[linecolor=black, dotstyle=o, dotsize=0.2, fillcolor=white](4.905,3.89)
\psdots[linecolor=black, dotstyle=o, dotsize=0.2, fillcolor=white](8.505,3.89)
\psdots[linecolor=black, dotstyle=o, dotsize=0.2, fillcolor=white](10.905,3.89)
\psdots[linecolor=black, dotstyle=o, dotsize=0.2, fillcolor=white](13.305,3.89)
\psdots[linecolor=black, dotstyle=o, dotsize=0.2, fillcolor=white](14.505,3.89)
\psdots[linecolor=black, dotstyle=o, dotsize=0.2, fillcolor=white](16.905,3.89)
\rput(6.905,-2.11){\Large $\tau$}
\rput(0.105,3.89){\Large $\underline n$}
\rput(0.105,1.09){\Large $\underline k$}
\rput(0.505,2.69){\large $\alpha$}
\psline[linecolor=black, linewidth=0.04, arrowsize=0.05291667cm 3.0,arrowlength=4.0,arrowinset=0.0]{->}(0.105,3.49)(0.105,1.49)
\end{pspicture}
}
\]
\caption{\label{Opet_jsem_podlehnul.}
Constructing $\alpha_v : \underline 7 \epi
  \underline 2$ out of $\alpha : \underline{15} \epi \underline 9$,
a tree $\tau \in \Tr^5(\alpha)$ and $v \in \vert(\tau)$.
The segment $\underline 9_v$ is marked by the dashed oval, the
elements of the set $\underline{15}_v$ by big punctured dots and the
elements of $\In(v)$ by two numbered balloons.  The map
$\alpha_v$ sends $2,3,4,7$ to $1$ and $1,5,6$ to $2$.
}
\end{figure}
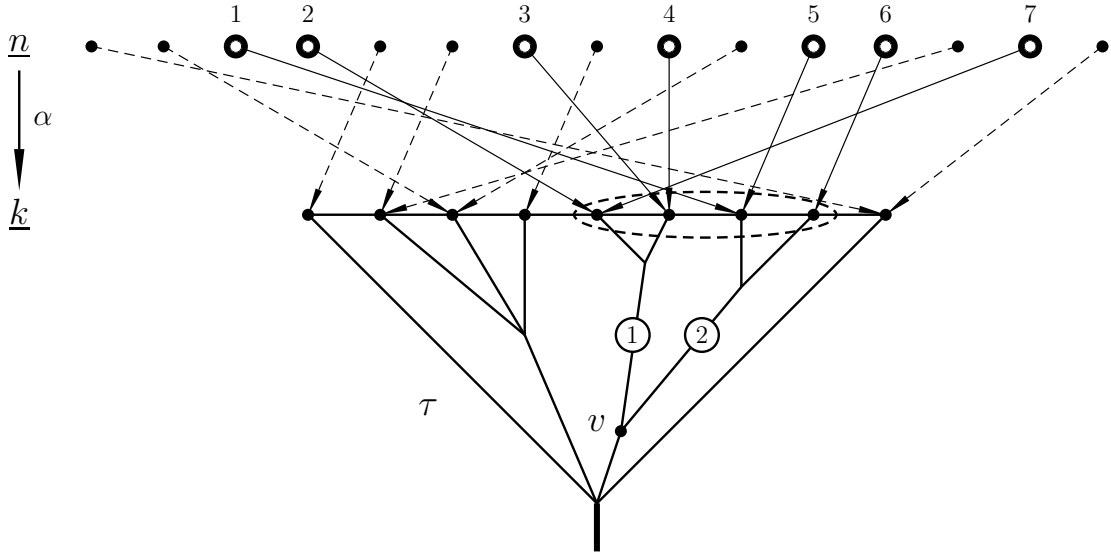
One puts
\[
\F(E)(\alpha) = \bigoplus_{s \geq 0}\F^s(E)(\alpha)
\]
with
\[
\F^s(E)(\alpha) := \bigoplus_{\tau \in \Tr^s(\alpha)} \bigotimes_{v \in
  \vert(\tau)}
E(\alpha_v).
\]

For each morphism $f : \alpha' \to \alpha''$ in $\ttP$ with fibers
$\Rada f1{k''}$, a
straightforward calculation reveals the existence of 
the canonical isomorphism
\[
\F(\alpha'') \ot \F(f_1) \ot \cdots \ot \F(f_{k''}) 
\stackrel\cong\longrightarrow  \F(\alpha')
\]
which we take as the structure operation~\eqref{dva_dny_za_sebou} of
the operad $\F(E)$. Notice that if
$E = \des^*(\underline E)$ for some $\underline E \in \Collectord$,
the above general construction coincides with the special one given in
Proposition~\ref{Leopold}.

\begin{remark}
If the reader finds this remark confusing, he or she may safely ignore
it. The $\ttP$-operad $\F(E)$ described above is the operad associated
to the free Markl's $\ttP$-operad under the isomorphisms of the categories
of ordinary and Markl's operads stated in~\hbox{\cite[Theorem~7.4]{Sydney}}
that holds due to the $1$-connectivity assumption; 
cf.~also the notes at the beginning of
Section~\ref{Zase_mam_chripku!}. 
\end{remark}

\begin{example}
\label{Koupil jsem si obtahovaci kamen.}
It is easy to see that 
$\F^0(E)$ is as in~\eqref{Zitra_sraz_s_Nimou.} and $\F^1(E) \cong
E$. To describe~$\F^2(E)$
we call, following~\cite[Definition~2.9]{Sydney}, a map $f:  \alpha \to \beta
\in \ttP$  {\em elementary\/} if all its fibers except precisely
one, say the $i$th fiber~$F$, are trivial, i.e.\ the chosen terminal
objects. We express this situation
by writing $F \fib_i \alpha \stackrel f\to \beta$ or simply $F \fib \alpha
\stackrel f\to \beta$ when $i$ is understood.
We leave as an easy exercise to verify that, with this notation, 
\begin{equation}
\label{Volal_Mikes.}
\F^2(E)(\alpha) = \bigoplus_
{F \fib \alpha \stackrel f\to \beta}
E(\beta) \otimes E(F),
\end{equation}
where the summation runs over all elementary maps $f: \alpha \to
\beta$ with $|\beta| \geq 2$.
\end{example}

\section{Minimal model of the terminal $\ttP$-operad and
  homotopy permutads}
\label{Pujdeme_dnes_s_Jarcou_na_muslicky?}

\begin{definition}
A {\em minimal model\/} of  an operad $\oP \in \Alg$ is 
a differential graded (dg)
$\ttP$-operad $\min = (\min, \pa)$
together with a dg $\ttP$-operad morphism $\rho : \min \to \oP$ such
that 
\begin{itemize}
\item [(i)]
the component $\rho(\alpha) : (\min(\alpha),\pa) \to
(\oP(\alpha),\pa=0)$ of $\rho$ is a homology isomorphism for each $\alpha \in
\ttP$, and
\item [(ii)]
the underlying non-dg $\ttP$-operad of $\min$ is free, and the differential
$\pa$ has no constant and linear terms with respect to the natural
grading of $\min$ (the {\em minimality condition\/}).
\end{itemize}
\end{definition}

Minimal models should be particular cofibrant replacements in a
conjectural (semi)model structure on the category of
$\ttP$-operads. For the purposes of applications it however suffices
to realize that minimal models are `special'
cofibrant~\hbox{\cite[Definition~17]{markl:ha}}. This already
guarantees that their algebras are homotopy invariant concepts.

Let $\underline a : \ainf \to \uAss$ be the minimal model 
of the terminal non-$\Sigma$ operad
$\uAss$ governing associative algebras. 
As we know from~\cite[Example~4.8]{markl:zebrulka},   
$\ainf$ is generated by the collection
$\underline E$ defined~by
\begin{equation}
\label{Za_necely_tyden_lettim_do_Ciny.}
\underline E (k) := \Span(\xi_{k-2}),\ \deg(\xi_{k-2}) := k\!-\!2,\ k \geq 2, 
\end{equation}
with the differential acting on the generators by the formula
\[
\pa(\xi_r) = \sum (-1)^{(b+1)(i+1)+b}\cdot \xi_a \circ_i \xi_b, \ r \geq 0,
\]
where the summation runs over all $a,b\geq 0$ with $a\+b=r\-1$,
$1\leq i\leq a\+2$, and where $\circ_i$ are the standard partial compositions
in a classical operad.
The dg operad morphism $\underline a : \ainf \to \uAss$
is given by 
\[
\underline a (\xi_r) := 
\begin{cases}
\mu \in \uAss(2) & \hbox {if $r=0$, and}
\\
0 & \hbox {otherwise,}
\end{cases}
\]
where $\mu \in \uAss(2)$ is the generator of $\uAss$.
Notice that $\Pterm \cong \des^*(\uAss)$.
Define finally $\min := \des^*(\ainf)$, and $\rho: \min \to \Pterm$ by
$\rho := \des^*(\underline a)$.

\begin{theorem}
\label{Vcera_jsem_se_vratil_z_Ciny.}
The dg $\ttP$-operad map $\rho : \min \to \Pterm$ defined above is a
minimal model of the terminal operad $\Pterm$.
\end{theorem}

\begin{proof}
The claim is almost obvious, but we still want to give some details,
namely a formula for the differential $\pa$ in $\min$. 
Given $\alpha : \underline n \epi
\underline k \in \ttP$, $n \geq k \geq 1$, one has, by definition,
\[
\min(\alpha) = \des^*(\ainf)(\alpha) = \ainf(k)
\]
while
\[
\Pterm(\alpha) \cong \des^*(\uAss)(\alpha) = \uAss(k).
\]
Under these identifications, the component $\rho(\alpha) : \min(\alpha) \to
\Pterm(\alpha)$ of the map $\rho$ equals 
\[
\underline 
a(k) : \ainf(k) \longrightarrow \uAss(k),
\] 
which is a homology isomorphism since $\ainf$ is the minimal model of~$\uAss$.
As a non-dg $\ttP$-operad, $\min$ is free, generated by the collection $E$
defined by
\[
E(\alpha) := \Span(\xi_{k-2}),\ \hbox { for }
\alpha : \underline n \epi \underline
k,\
k \geq 2,
\]
where the $\xi$'s are the same as in~\eqref{Za_necely_tyden_lettim_do_Ciny.}. 

Let us denote by $\xi_\alpha$ the replica of  $\xi_{k-2}$ in
$E(\alpha)$ above. Our next task will be to describe $\pa( \xi_\alpha) \in
\min(\alpha)$. For natural numbers $k, a,b \in \bbN$ such that
\begin{equation}
\label{Mark_nekomunikuje.}
k = a+b+3 \geq 2
\end{equation}
we define the map
$\f^{a,b}_i : \underline k \to \underline {a\!+\!2} \in
\Ord$ by the formula
\[
\f^{a,b}_i(j) : =
\begin{cases}
j & \hbox {for $1 \leq j \leq i-1$}
\\
i & \hbox {for $i \leq j \leq i+b+1$, and}   
\\
j\! -\!b\! -\!1  & \hbox {for $i+b+2 \leq j\leq k$.}   
\end{cases}
\]
Loosely speaking, 
$\f^{a,b}_i$ shrinks the interval $\{i,\ldots,i\!+\!b\!+\!1\} \subset  \underline
k$ to $\{i\}$. 

Let $\alpha : \underline n \epi \underline k\in \ttP$. By the
opfibration property of $\des : \ttP \to \Ord$, 
there exists a unique $\alpha^{a,b}_i : \underline n \epi    
\underline {a\!+\!2}  \in \ttP$  and a unique morphism  $f^{a,b}_i : \alpha \to
\alpha^{a,b}_i$ such that $\des(f^{a,b}_i) = \f^{a,b}_i$. Explicit
formulas for  $\alpha^{a,b}_i$ and  $f^{a,b}_i$ can be given~easily.

All fibers of $\f^{a,b}_i$ are trivial (i.e.\ unique surjections
to $\underline 1$) except the $i$th one which we denote by 
$F^{a,b}_i$.
Since $\min(U) \cong \F(E)(U) \cong \bbk$ for trivial $U$'s as
in~\eqref{Zprava_na_grantu.}, the structure
map~\eqref{dva_dny_za_sebou} 
for $\min$ corresponding to
$\f^{a,b}_i$ reduces to the `partial composition'
\[
\circ_i : \min(\alpha^{a,b}_i) \ot \min(F^{a,b}_i)
\longrightarrow \min(\alpha).
\]
The formula for the differential then reads
\begin{equation}
\label{Mam_porad_rymu.}
\pa(\xi_\alpha) = \sum (-1)^{(b+1)(i+1)+b}\cdot \xi_{\alpha^{a,b}_i} 
\circ_i
\xi_{F^{a,b}_i},
\end{equation}
where the summation runs over $k, a,b \in \bbN$ as
in~\eqref{Mark_nekomunikuje.} and $\xi_{\alpha^{a,b}_i}$
(resp.~$\xi_{F^{a,b}_i}$) are the replicas of $\xi_a$
(resp.~$\xi_b$) in $E(\alpha^{a,b}_i)$ (resp.~in $E(F^{a,b}_i)$). 
Formula~\eqref{Mam_porad_rymu.} makes the quadraticity of the
differential $\pa$ in $\min$, and therefore
its minimality in particular, manifest. 
\end{proof}

Formula~\eqref{Mam_porad_rymu.} can be written in a more intelligent
way.  Recall from Example~\ref{Koupil jsem si obtahovaci kamen.} that 
$F \fib_i \alpha \stackrel f\to \beta$ expresses that $f$ is an
elementary morphism whose only nontrivial fiber $F$ is the $i$th one. 
One then may
rewrite~\eqref{Mam_porad_rymu.} as
\begin{equation}
\label{Mam_furt_rymu.}
\textstyle
\pa(\xi_\alpha) = \sum_{F \fib_i \alpha \stackrel f\to \beta}
 (-1)^{(|F|+1)(i+1)+|F|}\cdot \xi_{\beta} 
\circ_i
\xi_{F},
\end{equation}
where the summation runs over all elementary maps $F \fib_i \alpha
\stackrel f\to \beta$ such that $|\beta| \geq 2$.
Following the philosophy of~\cite[Section~4]{markl:zebrulka}, 
we formulate

\begin{definition}
A {\em strongly
homotopy permutad\/} is an algebra for the minimal model $\min$ of~$\Pterm$.
\end{definition}

\begin{remark}
Notice that to both~\eqref{Mam_porad_rymu.} and~\eqref{Mam_furt_rymu.}
only partial compositions
enter. This is because $\min$ is isomorphic to the
dual dg operad of the Koszul dual of $\Pterm$, 
cf.\  the notes at the beginning
of Section~\ref{Zase_mam_chripku!} and Corollary~\ref{Zavolam_Jarce?}.
\end{remark}

Strongly homotopy permutads can be described directly via their
structure operations:

\begin{proposition}
\label{Klasicka_poprijezdova_deprese}
A strongly homotopy permutad is a collection $A = \coll An$ of dg vector
spaces together
with structure maps
\[
\pi_\alpha : A(\underline n_1) \ot \cdots \ot A(\underline n_k) \longrightarrow
A(\underline n)
\]
of degree $k\!-\!2$ defined for each $\alpha : \underline n \to \underline
k$, $n \geq k \geq 2 \in \ttP$; here $\underline n_i := \inv
\alpha(i)$, $1 \leq i \leq k$.
Moreover, for each such an $\alpha$, the equality
\[
\tag{$P_\alpha$}
\textstyle
\pa(\pi_\alpha) = \sum_{F \fib_i \alpha \stackrel f\to \beta}
 (-1)^{(|F|+1)(i+1)+|F|}\cdot \pi_{\beta} 
\circ_i
\pi_{F}
\]
is satisfied. Here the summation is the same as
in~\eqref{Mam_furt_rymu.}, $\pi_{\beta} 
\circ_i \pi_{F}$ is the multilinear function obtained by inserting
$\pi_{F}$ into the $i$th slot of $\pi_{\beta}$, and $\pa$ in the left
hand side is the differential on the endomorphism complex 
induced by the differential of
$\coll An$.   
\end{proposition}

\begin{remark}
It can be easily checked that strongly homotopy permutads are algebras
for the minimal model of the associative operad evaluated in
the endomorphism operad that uses the shuffle product of vector
spaces instead of the usual one. The similarity of~($P_\alpha$) with a
formula in~\cite[Example~4.8]{markl:zebrulka} is therefore not surprising.
\end{remark}

\begin{example}
\label{miliarda_komunistu=Cina}
If $|\alpha| = 2$, the sum in~($P_\alpha$) is empty, 
thus $\pa(\xi_\alpha)=0$,  so $\pa(\xi_\alpha)$ is a dg map.
If $\alpha : \underline m \epi \underline 3$, 
the sum in~($P_\alpha$) has two terms, corresponding
to the two possible order-preserving surjections $\underline 3 \epi
\underline 2$. The associated elementary morphisms are
\[
v \fib_1 \alpha \longrightarrow u \ \hbox { and } s \fib_2 \alpha \longrightarrow t, 
\]  
where $u,v,s,t \in \ttP$ are as in
Definition~\ref{V_nedeli_letim_do_Moskvy.}. Axiom~($P_\alpha$) now
takes the form
\[
\pa(\pi_\alpha) = \pi_u \circ_1 \pi_v -  \pi_t \circ_2 \pi_s =
\pi_u(\pi_v \ot \id) - \pi_t(\id \ot \pi_s).
\]
The degree $0$ operations $\pi_r$ for $r \in \ttP$ with $|r|=2$ are
therefore of the same type as the operations $\pcirc_r$ of
Definition~\ref{V_nedeli_letim_do_Moskvy.}, but they satisfy the
`associativity' \eqref{Vratim_se_z_te_Ciny?} only up to the 
homotopy~$\pi_\alpha$. For $|\alpha|=4$, ($P_\alpha$) is a
permutadic version of the Mac Lane's pentagon.
\end{example}

\begin{remark}
\label{V_Koline_jsem_povesil_lustr.}
Let $\K = \{\K_n\}_{n \geq 1}$ be the cellular topological non-$\Sigma$ operad of
the Stasheff's
associahedra~\cite[II.1.6]{markl-shnider-stasheff:book}. 
Then $\M := \des^*(\K)$ is the cellular topological $\ttP$-operad such
that the minimal model $\min$ of $\Pterm$
is isomorphic to the associated cellular chain  $\ttP$-operad
$\CC_*(\M)$. This is where the `hidden associahedron' of the
title of this article hides.
\end{remark}

\begin{question}
The classical recognition theorem~\cite{jds:hahI} states that a  connected
topological space has a weak homotopy type of a based loop space if and only
if it is a topological $\K$-algebra (aka $A_\infty$-space). Does
there exist an analogous statement for topological $\M$-permutads? 
\end{question}

As in the case of $A_\infty$-algebras, the left hand side
of~($P_\alpha$) can be absorbed into the right one
by interpreting $\pa$ as a structure operation. This can be done
by allowing $\pi_\alpha$'s also for $|\alpha| = 1$, i.e.\ for $\alpha$
the terminal object $U_n :\underline n \epi \underline 1$, $n \geq 1$. The
associated structure map then will be a degree $-1$ linear morphism
\[
\pi_n := \pi_{U_n} : A(\underline n) \longrightarrow  A(\underline n).
\]  
Further, we need to allow in the sum of~($P_\alpha$) also
trivial $F$ or $\beta$. The modified axiom reads
\begin{equation}
\tag{$P'_\alpha$}
\textstyle
0 = \sum_{F \fib_i \alpha \stackrel f\to \beta}
 (-1)^{(|F|+1)(i+1)+|F|}\cdot \pi_{\beta} 
\circ_i
\pi_{F}
\end{equation}
For $\alpha = U_n$ it clearly gives $\pi_n^2  = 0$, 
thus $\pi_n$ is a degree $-1$
differential of $A(\underline n)$. For a~general~$\alpha$, the left
hand side of ($P_\alpha)$ is absorbed in the right hand side
of~($P'_\alpha$)
in terms with $|F| = 1$ or $|\beta| = 1$. We leave the details as an
exercise.

\section{Koszul duals of $\ttP$-operads}
\label{Slepim_si_i_410?}

A `classical' operad is binary quadratic if it is generated by 
operations of arity~$2$, and its ideal of relations is
generated by relations of arity $3$. Each such an operad admits its
Koszul, aka quadratic, dual,
cf.~\cite[(2.1.9)]{ginzburg-kapranov:DMJ94}
or~\cite[Definitions~II.3.31 and
II.3.37]{markl-shnider-stasheff:book}. Similar notions
exist for operads in a general operadic 
category~\cite[Section~11]{Sydney}. In the remaining two sections we will
analyze the particular case of operads in $\ttP$. The floor plan is similar
to that of the parallel theory for permutads presented in Sections~\ref{2+3}
and~\ref{Pristi_tyden_mam_spoustu_uradovani.}, so we can afford to be
more telegraphic.  We start with

\begin{definition}
\label{Smutne_narozeniny.}
A $\ttP$-operad $\oP$  is {\em binary quadratic\/} if it is of the
form  
$\oP \cong \F(E)/(R)$, where
\begin{itemize}
\item [(i)]
the generators $E = \{E(\alpha)\}_{\alpha \in \ttP}$ are such that
$E(\alpha) = 0$ if $|\alpha| \not=2$, and
\item [(ii)]
the generators $R$ of the ideal of relations form a subcollection of
$\F^2(E)$. 
\end{itemize}
\end{definition}

In Definition~\ref{Smutne_narozeniny.}, $\F(E)$ is the free
$\ttP$-operad generated by the collection $E$. In concrete examples
treated in the remainder of this section, $E$ will always be the
restriction
$\des^*(\underline E)$ of some $\underline E   \in \Collectord$,
thus it may be, by the virtue of Proposition~\ref{Leopold}, 
realized by the restriction $\des^* (\uF(\underline E))$ of the free
non-$\Sigma$ operad $\uF(\underline E)$. The following notion is
however recalled for an arbitrary $E$. 

The {\em Koszul dual\/}  $\oP^!$ of 
a binary quadratic $\ttP$-operad  
$\oP$~\cite[Definition~11.3]{Sydney} is the quotient
\[
\oP^!:=\F(\susp\, E^*)/(R^\perp), 
\]
where \hbox{$\susp\, E^*$} is the suspension of the  component-wise linear 
dual of the generating
$\ttP$-collection~$E$, and $R^\perp \subset \F^2(\susp E^*)$ is the
component-wise annihilator
of $R\subset \F^2(E)$ in the obvious degree~$-2$ pairing between
\[
\F^2(\susp\, E^*)(\alpha) = \bigoplus_
{F \fib \alpha \stackrel f\to \beta}
\susp\, E(\beta)^* \otimes \susp\, E(F)^*
\cong\ \susp^2 \hskip -.3em\bigoplus_
{F \fib \alpha \stackrel f\to \beta}
E(\beta)^* \otimes E(F)^*
\]
and
\[
\F^2(E)(\alpha) = \bigoplus_
{F \fib \alpha \stackrel f\to \beta}
E(\beta)\, \otimes E(F);
\]
here we use the explicit description of $\F^2(-)$ given in
Example~\ref{Koupil jsem si obtahovaci kamen.}. 
The following statement follows from
\cite[Proposition~14.4]{Sydney} but we will present a proof that uses our
explicit knowledge of the minimal model $\min$ of $\Pterm$ acquired in
Section~\ref{Pujdeme_dnes_s_Jarcou_na_muslicky?}. 

\begin{proposition}
\label{V_sobotu_letim_do_Sydney.}
The operad  $\Pterm$ is binary
quadratic. It is  self-dual in the sense that the category of algebras over
$\Pterm^!$ is isomorphic to the category of permutads via the functor
induced by the suspension of the underlying collection.
\end{proposition}

\begin{proof}
By Theorem \ref{Vcera_jsem_se_vratil_z_Ciny.}, $\Pterm \cong
H_0(\min)$. Since $\min$ is non-negatively homologically graded, its
suboperad  $Z_0(\min)$ of degree $0$ cycles equals the degree $0$
piece $\min_0$ of $\min$. Clearly $\min_0 = \F(E)_0 \cong \F(E_0)$ where, by the
definition of the generating collection $E$,
\[
E_0(\alpha) = 
\begin{cases}
\Span(\xi_0)  & \hbox {if $|\alpha| = 2$, and}
\\
0 & \hbox {otherwise.}   
\end{cases}
\]
Likewise, $\min_1 = \F(E)_1$ consists of compositions of some number
of elements of $E_0$ and precisely one element of $E_1$, where
\[
E_1(\alpha) = 
\begin{cases}
\Span(\xi_1)  & \hbox {if $|\alpha| = 3$, and}
\\
0 & \hbox {otherwise.}   
\end{cases}
\]

Thus $B_0(\min) = {\rm Im}(\pa : \min_1 \to \min_0)$ equals the
ideal generated by $\pa(\xi_\alpha)$, $|\alpha| = 3$, where
$\xi_\alpha$ is the replica of $\xi_1$ in $E(\alpha)$.
Using formula~\eqref{Mam_furt_rymu.} for the differential, we conclude that
\[
\Pterm \cong \F(E_0)/(\xi_u \circ_1 \xi_v - \xi_t \circ_2 \xi_s),
\]
where $u,v,s$ and $t$ have the same meaning as in Example
\ref{miliarda_komunistu=Cina}. This is the required binary quadratic
presentation of the terminal $\ttP$-operad.
The natural  pairing 
\[
\F^2(\susp E^*) \ot \F^2(\susp E)^* \longrightarrow \bbk
\]
is given by 
\[
\langle \xi_a \circ_i \xi_b\ |\ \xi_c^\uparrow \circ_j \xi_d^\uparrow
\rangle
:=
\begin{cases}
1 \in \bbk &\hbox{if $a=c$, $b=d$ and $i=j$, while}
\\
0 &\hbox{otherwise.}   
\end{cases}
\]
In the above formula, the up-going arrow indicates the corresponding
suspended generator. 
One thus immediately gets 
\[
\Pterm^! \cong \F(\susp E_0)/
(\xi_u^\uparrow \circ_1 \xi_v^\uparrow + \xi_t^\uparrow \circ_2 \xi_s^\uparrow).
\]
As in the proof of Proposition~\ref{Treti_den_je_kriticky.} we verify
that $\Pterm^!$-algebras are collections $A = \coll An$ equipped with 
degree $+1$ operations 
\[
\pcirc_r^\uparrow : A(\inv r(1)) \ot  A(\inv r(2)) \longrightarrow A(\underline n)
\]
satisfying the `anti-associativity'
\[
\pcirc_u^\uparrow(\pcirc_v^\uparrow \ot \id) +
\pcirc_t^\uparrow(\id \ot \pcirc_s^\uparrow) = 0.
\]

Let  \hbox{$\susp A := \coll{\susp A}n$} be the component-wise suspension of
the collection $A$. It is straightforward to see that the operations
\[
\pcirc_r :
\susp A(\inv r(1)) \ot  \susp A(\inv r(2)) \stackrel{\desusp \ot
  \desusp}\longrightarrow 
 A(\inv r(1)) \ot  A(\inv r(2)) \longrightarrow A(\underline n)  
\stackrel{\susp}\longrightarrow \susp A(\underline n) 
\] 
make $\susp A$ a permutad. The correspondence \hbox{$(A,\pcirc_r^\uparrow)
\mapsto  (\uparrow A,\pcirc_r)$} is clearly an isomorphism between the
categories of $\Pterm^!$-algebras and $\Pterm$-algebras.
One may in fact show that  the `operadic desuspension'
${\mathbf s}^{-1}: \Pterm^! \to \Pterm$ defined by ${\mathbf s}^{-1} (\alpha): =\
\desusp^{|\alpha| -2}$  is an isomorphism of $\ttP$-operads.
\end{proof}

The following proposition will be formulated for 
any strict operadic
functor~\cite[p.~1635]{duodel} between 
arbitrary operadic categories, but the reader might as
well consider only the case of $\des: \ttP \to \Ord$.

\begin{proposition}
\label{Jarka_mi_pripravuje_smutne_narozeniny.}
Let  $p:\ttO\to {\tt P}$ be a strict operadic functor,
$\oP$ a ${\tt P}$-operad and $\J \subset \oP$ an ideal. Then
\begin{itemize}
\item [(i)]
the restriction $p^*(\J)$ is an ideal in the $\ttO$-operad $p^*(\oP)$.
\item [(ii)]
If $\J$ is generated by $G \subset \J$, then  $p^*(\J)$ is
generated by the restriction $p^*(G)$. Finally, 
\item [(iii)]
the restriction $p^*(\oP/\J)$ of the quotient $\oP/\J$ is isomorphic
to $p^*(\oP)/p^*(\J)$.
\end{itemize}
\end{proposition}

\begin{proof}
Items (i) and (ii) easily follow from the basic properties of operads
in operadic categories and their ideals. Let us establish item
(iii). Since the restriction $p^*$ 
is a functor \hbox{$\Op{\tt P} \to \Op\ttO$} between the categories of
operads~\cite[p.~1639]{duodel},
applying it to the projection $\pi : \oP \epi \oP/\J$ leads to an
operad morphism
\[
p^*(\pi) : p^*(\oP) \longrightarrow p^*(\oP/\J),
\] 
which is clearly surjective.

We want to prove that $\Ker(p^*(\pi)) \cong p^*(\J)$.
Assume that $u \in p^*(\oP)(t)$ for some $t \in \ttO$ is such that $p^*(\pi)(u) =
0$. By the definition of the restriction, $p^*(\oP)(t) = \oP(T)$ with
$T := p(t)$ and, likewise,  
$p^*(\oP/\J)(t) = (\oP/\J)(T) = \oP(T)/\J(T)$. Under this
identification, $p^*(\pi)$ acts as the projection  
$\oP(T) \epi (\oP/\J)(T) = (\oP(T)/\J(T))$, so $u \in \J(T)) = p^*(\J)(t)$.
\end{proof}

The following corollary can be stated for any 
discrete opfibration $p:\ttO\to {\tt P}$, but for our purposes
the case of $\des : \ttP \to \Ord$ for which we recalled all the relevant
notions will be sufficient. 

\begin{proposition}
\label{Lepim_si_L-410.}
Let $\underline \oP$ be a binary quadratic non-$\Sigma$-operad. Then
$\des^*(\underline \oP)$ is binary quadratic and, moreover,
$\des^*(\underline \oP)^! \cong \des^*(\underline \oP^!)$
\end{proposition}

\begin{proof}
Suppose that $\underline \oP = \uF(\underline  E)/(\underline R)$ is a binary
quadratic non-$\Sigma$ operad. Then, by
Proposition~\ref{Leopold}, $\des^*(\uF(\underline  E))$ is the free
$\ttP$-operad $\F(E)$ on the restriction $E := \des^*(\underline E)$ thus, by
Proposition~\ref{Jarka_mi_pripravuje_smutne_narozeniny.}, 
\[
\des^*(\underline \oP) \cong \des^*(\uF(\underline  E))/(\des^*\underline
R)
= \F(E)/(R),
\]
with $R := \des^*(\underline R)$. It is clear from definitions that
$\F(E)/(R)$ is a binary quadratic presentation of the $\ttP$-operad $\oP
:= \des^*(\underline \oP)$.
The second part of the proposition follows from
the canonical isomorphisms 
\[
\des^*(\susp\, \underline E^*) \cong\ \susp\,( \des^*(\underline E))^*
\ \hbox { and } \
\des^*(\underline R^\perp) \cong (\des^*(\underline R))^\perp
\]
which can be checked directly.
\end{proof}

\section{Koszulity of $\ttP$-operads}
\label{Zase_mam_chripku!}

Operads are classically presented as collections  $\uoS =
\{\uoS(n)\}_{n\geq 1}$  of dg vector spaces
with structure operations
\begin{equation}
\label{Dostal jsem obrovskou lahev rumu.}
{\underline m}_{\, \Rada n1k} : \uoS(k) \ot \uoS(n_1) \ot \cdots \ot \uoS(n_k)
\longrightarrow 
\uoS(n_1 \+ \cdots \+ n_k)
\end{equation} 
specified for any $k, \Rada n1k \geq 1$. One also assumes 
a choice of the unit map \hbox{$\eta : \bfk \to  \uoS(1)$}. 
In \cite[Def.~1.1]{markl:zebrulka} the author
introduced an alternative definition, replacing the operations
in~\eqref{Dostal jsem obrovskou lahev rumu.} by the bilinear ones of the
form
\begin{equation}
\label{letel_jsem_ve_snehovych_prehankach}
\circ_i : \uoS(k) \ot \uoS(n) \longrightarrow \uoS(k\+n\-1), \ n \geq
1,\
1 \leq i \leq k.
\end{equation} 

In this setup, the grading of free operads is given 
by the number of structure operations applied to the generators minus
one, so the operations
in~\eqref{letel_jsem_ve_snehovych_prehankach} naturally assemble into
a~single linear map
\begin{subequations}
\begin{equation}
\label{Co to mam s tlakem?}
\circ: \uF^2(\uoS) \to \uoS
\end{equation}
from the grade-$2$ part of the free operad $\uF(\uoS)$ generated by $\uoS$ to
$\uoS$. Likewise, for a cooperad $\uoC$ with structure operations of
the form
\[
\Delta_i : \uoC (k\+n\-1) \longrightarrow \uoC(k) \ot \uoC(n),\ \ k,n \geq
1,\
1 \leq i \leq k, 
\]
 dual to~(\ref{letel_jsem_ve_snehovych_prehankach}), one has a naturally induced map
\begin{equation}
\label{dnes ohnicek s detmi}
\underline \Delta : \uoC \to \uF^2(\uoC).
\end{equation}
\end{subequations}
This enables one to define (co)bar constructions of (co)operads
by mimicking the analogous classical constructions for
(co)algebras, cf.~also the constructions for permutads in
Section~\ref{Pristi_tyden_mam_spoustu_uradovani.}. 

Partial compositions $\circ_i$ are expressed via the operations 
in~\eqref{Dostal jsem obrovskou lahev rumu.} as
\begin{subequations}
\begin{equation}
\label{dve hodiny na Blaniku}
p \circ_i q := {\underline m}_{\,1,\ldots,n,\ldots ,1}\big(p \ot e^{\ot (i-1)} \ot q \ot e^{\ot (k-i)}\big),\ 
e := \eta(1), \ p \in \uoS(k), \ q \in \uoS(n).
\end{equation}
An element-free definition of $\circ_i$ is provided by the diagram
\begin{equation}
\label{Jsem s Fukem v Koline.}
\xymatrix@C=3em{\uoS(k) \ot \uoS(n)  \ar[r]^(.35)\cong&\ar[d]
 \uoS(k) \ot \bfk\ot \cdots \ot \uoS(n) \ot
\cdots  \ot \bfk
\\
& \uoS(k) \ot \uoS(1) \ot \cdots \ot \uoS(n) \ot
\cdots  \ot \uoS(1)\  \ar[r]^(.69){ {\underline m}_{1,\ldots,n,\ldots ,1}}
&\ \uoS(k\+n\-1)
}
\end{equation}
\end{subequations}
whose vertical map is the product of the identities and
the unit morphisms $\eta : \bfk \to \uoS(1)$.
For classical unital operads, both approaches are equivalent -- 
the operations in~\eqref{Dostal jsem obrovskou lahev rumu.} can be
expressed via the $\circ_i$-operations as
\[
{\underline m}_{\,\Rada n1k}(a \ot b_1 \ot \cdots \ot b_k) :=
{\varepsilon}\cdot (\cdots ((a \circ_k b_k)\circ_{k-1} b_{k-1})\circ_{k-2}
\cdots b_2)\circ_1 b_1,
\]
where $a \in \uoS(k)$, $b_i \in \uoS(n_i)$ for $1 \leq i
\leq k$, and $\varepsilon = \pm 1$ is the Koszul sign of the permutation $(\Rada b1k) 
\mapsto (\Rada bk1)$. Operad-like
structures based on `partial compositions' as
in~\eqref{letel_jsem_ve_snehovych_prehankach} were later called
Markl's operads.  

The arrays
$\Rada n1k$ indexing the operations 
in~\eqref{Dostal jsem obrovskou lahev rumu.} are in  
one-to-one correspondence 
with order-preserving epimorphisms $\gamma:   \underline
m \epi \underline k \in \Ord$; such a $\gamma$ determines 
the sequence $\Rada n1k$ in
which $n_j$ is the cardinality of the set-theoretic preimage
$\gamma^{-1}(j)$, $1 \leq j \leq k$. Notice that the definition of
$\circ_i$ given in~\eqref{dve hodiny na Blaniku} or~\eqref{Jsem s
  Fukem v Koline.} uses $\gamma$'s of special types only, namely those for
which the cardinality of the preimage $\gamma^{-1}(i)$ is $> 1$ only
for $j=i$. Such $\gamma$'s are prototypes of elementary maps in
operadic categories.

Operads in general operadic categories also exist in two
disguises which are, under favorable conditions, equivalent -- 
in a form where the compositions in all inputs are
made simultaneously; this is how they were introduced
in~\cite{duodel} -- and in Markl's form where
they are  performed one after one. 
The crucial advantage of Markl's form is, as in the classical case,
that free Markl's operads are naturally graded by the length of
the chain of compositions and that the structure operations assemble
into maps of the form~(\ref{Co to mam s tlakem?}) for operads resp.~(\ref{dnes
  ohnicek s detmi}) for cooperads.
This will be used in the definition of 
the central object of this section,
the dual dg operad of a $1$-connected
$\ttP$-operad~$\oP$.

Recall from Example~\ref{Koupil jsem si obtahovaci kamen.} 
that a map $f:  \alpha \to \beta  
\in \ttP$  in~\eqref{nikdo_mi_nepopral_k_narozeninam} 
is elementary if all its fibers except precisely
one, say the $i$th fiber~$F$, are trivial, i.e.\ equal to some
local terminal objects $U_{n_s}$, $s \not= i$. 
This fact was recorded by   $F \fib \alpha \stackrel f\to \beta$.
For each such an $f$ we
define, mimicking~\eqref{Jsem s Fukem v Koline.},  
the {\em partial composition\/} 
\[
\circ_f : 
\oP(\beta) \ot \oP(F)  \to \oP(\alpha),\ F \fib \alpha \stackrel
f\longrightarrow \beta, 
\] 
as the composite
\begin{equation}
\label{Zavola?}
\xymatrix@C=3em{\oP(\beta) \ot \oP(F)  \ar[r]^(.35)\cong&\ar[d]
 \oP(\beta) \ot \bfk\ot \cdots \ot \oP(F) \ot
\cdots  \ot \bfk
\\
& \oP(\beta) \ot \oP(U_{n_1}) \ot \cdots \ot \oP(F) \ot
\cdots  \ot \oP(U_{n_{k''}})  \ar[r]^(.8){m_f}
& \oP(\alpha)
}
\end{equation}
in which $m_f$ is the structure map~\eqref{dva_dny_za_sebou} and 
the vertical map is the product of the identities and isomorphisms 
$\oP(U_{n_s}) \cong \bbk$, $s \not=i$,  which follow from the
$1$-connectivity of $\oP$. The partial compositions satisfy
appropriate axioms~\cite[Definition~7.1]{Sydney} derived 
from the properties of $m_f$'s. The structure described above is
Markl's version of a $\ttP$-operad $\oP$.
Under the $1$-connectivity assumptions, the
categories of $\ttP$-operads and Markl's $\ttP$-operads
coincide~\hbox{\cite[Theorem~7.4]{Sydney}}.

A Markl's {\em $\ttP$-cooperad\/} is a collection  $\cC =
\{\cC(\alpha\}_{\alpha \in \ttP}$ with operations
\begin{equation}
\label{Popousel_jsem_si_nabrousit_britvu_na_kameni.}
\Delta_f : \oC(\alpha) \longrightarrow
\oC(\beta) \ot \oC(F),\ F \fib \alpha \stackrel f\longrightarrow \beta,
\end{equation}
satisfying the  dual versions of Markl's
operads~\cite[Definition~7.1]{Sydney}. 
As in Example~\ref{Pisu_v_Koline_po_skoleni_ze_stavby.}, under obvious 
finitarity assumptions, the component-wise linear dual $\oP^*$ of a
Markl's $\ttP$-operad is a Markl's $\ttP$-cooperad. We will need also
a reduced version of $\cC$ whose underlying $\ttP$-collection is
defined by
\[
\ocC(\alpha) := 
\begin{cases}
0&\hbox{if  $|\alpha| = 1$, and}
\\
\cC(\alpha)&\hbox{if  $|\alpha| \geq 2$,}
\end{cases}
\]
and its structure operation
\begin{equation}
\label{Strasliva_zima_uz_neni.}
\oDelta_f : \ocC(\alpha) \longrightarrow
\ocC(\beta) \ot \ocC(F)
\end{equation}  
equals $\Delta_f$
in~\eqref{Popousel_jsem_si_nabrousit_britvu_na_kameni.} if $|\beta|
\geq 2$ while for $|\beta| = 1$ it is the zero map
\[
\ocC(\alpha) \longrightarrow
\ocC(\beta) \ot \ocC(F) =0.
\]
In fancy language, $\ocC$ is the coaugmentation coideal in Markl's
cooperad $\cC$.
It follows from~\eqref{Volal_Mikes.} that the individual structure
operations~\eqref{Strasliva_zima_uz_neni.} assemble into a single map
\begin{equation}
\label{Zdalo_se_mi_o_Zebrulce.}
\oDelta : \ocC \cong \F^1(\ocC) \longrightarrow  \F^2(\ocC).
\end{equation}

\begin{definition}
A degree $s$ linear map $\varpi
: \oP \to \oP$ of $\ttP$-collections
is a degree $s$ {\em derivation\/} of a Markl's operad  $\oP$ 
if
\[
\varpi\,\circ_f = \circ_f (\varpi \ot \id) +  \circ_f(\id \ot
\varpi),
\]
for every elementary 
$F \fib \alpha \stackrel f\to \beta$ and the associated operation
$\circ_f  : 
\oP(\beta) \ot \oP(F)  \to \oP(\alpha)$.
\end{definition}

One easily sees that, given a
$1$-connected $\ttP$-collection, each 
degree $s$ linear map of $\ttP$-collections $\zeta :E \to \F(E)$ uniquely
extends to a degree $s$ derivation $\varpi$  of the free $\ttP$-operad
$\F(E)$.
For a Markl's $\ttP$-cooperad $\cC$ one has
a degree $-1$ map $\zeta :\ \desusp \ocC  \to  \F^2(\desusp \ocC)$ of
$\ttP$-collections 
defined as the composition
\begin{equation}
\label{Za_chvili_mi_zacina_seminar.}
\zeta :=\
\desusp \ocC \stackrel \uparrow\longrightarrow \ocC
\stackrel\oDelta\longrightarrow
\F^2(\ocC) \stackrel\cong\longrightarrow \,\uparrow^2 \F^2(\desusp \ocC)
\stackrel{\downarrow^2}\longrightarrow \F^2(\desusp \ocC)
\end{equation}
where $\oDelta$ is as in~\eqref{Zdalo_se_mi_o_Zebrulce.} and
\hbox{$\F^2(\ocC) \stackrel\cong\longrightarrow \,\uparrow^2\!
  \F^2(\desusp \ocC)$}  the
obvious canonical isomorphism
\[
\F^2(\ocC)
\cong  \bigoplus_
{F \fib \alpha \stackrel f\to \beta}
\ocC(\beta) \otimes \ocC(F)
 \stackrel\cong\longrightarrow\  \susp^2 \hskip -.3em  \bigoplus_
{F \fib \alpha \stackrel f\to \beta}
\desusp \ocC(\beta) \otimes \desusp \ocC(F) \cong   \,\uparrow^2\!
  \F^2(\desusp \ocC).
\]
Denote finally 
by $\pa_\cobar$ the unique extension of $\zeta$ into a degree $-1$
derivation of \hbox{$\F(\desusp \hskip -.05em \ocC)$}. 
One easily verifies that $\pa_\cobar^{\,2} = 0$.

As we noticed at the beginning of this section, under the
$1$-connectivity assumption, operads and Markl's operads are just
different presentations of the same objects, which is true also for
(Markl's) cooperads. 
We will therefore make no difference between them. Having this
in mind, we formulate

\begin{definition}
The {\em cobar construction\/}  of a
$\ttP$-cooperad $\cC$  is the dg
$\ttP$-operad ${\cobar} (\cC) := (\F(\desusp \ocC),\pa_\cobar)$. The {\em dual dg
  operad\/} of a $\ttP$-operad $\oP$  satisfying
appropriate finitarity assumptions  is the dg $\ttP$-operad
$\D(\oP) := \cobar(\oP^*)$.
\end{definition}

To introduce the Koszulity of a binary quadratic
$\ttP$-operad $\oP$, 
one starts from an injection $\susp E \hookrightarrow  \oP^!$ of
$\ttP$-collections defined as the composite
\[
\susp E \hookrightarrow \F(\susp E) 
\twoheadrightarrow \F(\susp E)/(R^\perp) = \oP^!.
\]
Its linear dual  ${\oP^!}^* \twoheadrightarrow\ \susp
E$ desuspens to a map $\pi:\ \desusp {\oP^!}^*
\twoheadrightarrow E$. As for permutads, one has the related twisting
morphism  $\desusp {\oP^!}^*  \to \oP$, which  is 
the composition
\[
 \desusp  {\oP^!}^*\stackrel\pi\twoheadrightarrow  E 
 \hookrightarrow \F(E) 
\twoheadrightarrow \F(E)/(R) = \oP.
\]
By the freeness of $\F(\desusp
{\oP^!}^*)$, it extends to a  morphism
$\rho : \F(\desusp {\oP^!}^*) \to \oP$  of dg $\ttP$-operads. 
One verifies by direct calculation:

\begin{proposition}
The morphism $\rho$ induces the {\em
  canonical map\/} 
\begin{equation}
\label{Jarce_prijede_M1_a_bude_do_soboty.}
\can:
\D(\oP^!) = (\F(\desusp {\oP^!}^*), \pa_\D)    \longrightarrow (\oP,0)
\end{equation}  
of dg $\ttP$-operads.
\end{proposition}

\begin{definition}
\label{hodnoceni snad odlozeno}
A binary quadratic $\ttP$-operad $\oP$ is {\em Koszul\/} if the 
canonical map~\eqref{Jarce_prijede_M1_a_bude_do_soboty.} is 
a~component-wise homology isomorphism.
\end{definition}

In the following theorem,  $\ucobar(\ucC)$ denotes the 
cobar construction of a $1$-connected 
classical non-$\Sigma$ 
cooperad\footnote{Notice  
that each $1$-connected cooperad is  coaugmented.} 
$\ucC$, i.e.\ a non-$\Sigma$ version of the construction 
in~\cite[Section~6.5.2]{loday-vallette},  
and $\uD(\uoP) : = \ucobar(\uoP^*)$  the dual dg
operad
of a $1$-connected non-$\Sigma$ operad
$\uoP$ with appropriate finitarity properties that make the dualization possible.

\begin{theorem}
\label{Je_teprve_utery_a_uz_padam_na_usta.}
Let $\ucC$ and $\uoP$ be as above, and $\cC := \des^*(\ucC)$, $\oP :=
\des^*(\uoP)$. Then 
\begin{itemize}
\item[(i)]
the dg $\ttP$-operads $\cobar(\oC)$ and $\des^*(\ucobar(\ucC))$ are isomorphic,
as they are
\item [(ii)]
the dg $\ttP$-operads $\D(\oP)$ and $\des^*(\uD(\uoP))$.
\item[(iii)]
Assume that $\uoP$ is binary quadratic. Then the
$\ttP$-operad $\oP$ is (binary quadratic)
Koszul if and only if  $\underline \oP$ is one.
\end{itemize}
\end{theorem}

\begin{proof}
The restriction along $\des : \ttP \to \Ord$ forms a functor
from the category of dg $\Ord$-collections to the category of dg 
$\ttP$-collections. We already noticed that the restriction
takes $\Ord$-operads, i.e.\ non-$\Sigma$ operads, to
$\ttP$-operads. The same is true for
Markl's operads and cooperads. The restriction also  commutes with
(de)suspensions, component-wise linear duals and homology.
By this we mean the following.

Recall that, if $\uX$ is a rule that assigns to 
$\underline n \in \Ord$ a dg vector space $\uX(\underline n)$,
$\des^*( \uX)$ assigns to $\alpha
\in \ttP$ a dg vector space $\uX(\des (\alpha))$. Then one has
\begin{subequations}
\begin{align}
\label{Mam prohlidku.}
\susp\, \des^*(\uX)  = \des^*(\susp\, \uX), & \
\desusp \des^*(\uX)  = \des^*(\desusp \uX),  
\\
\label{Musim koupit baterky.}
(\des^*(\uX))^* = \des^*(\uX^*),& \
H_*( \des^*(\uX)) = \des^*(H_*(\uX)).
\end{align}
Indeed, for $\alpha \in \ttP$ one has by definition 
\[
(\susp\, \des^*(\uX))(\alpha) =\ \susp\, \uX(\des (\alpha)) =
\des^*(\susp\, \uX)(\alpha),
\]
which gives the first equality of~\eqref{Mam prohlidku.}. The
verification of the remaining ones is similar. We will need one more
auxiliary statement. 

Let $\uE$ be a  $\Ord$-collection, $\underline\zeta : \uE \to \uF(\uE)$ a linear
map, and $\underline\varpi : \uF(\uE) \to \uF(\uE)$ the unique
extension of $\underline\zeta$ into a derivation of the free non-$\Sigma$
operad $\uF(\uE)$. Denote $E := \des^*(\uE)$, $\zeta :=
\des^*(\underline\zeta)$, 
and let $\varpi : \F(E) \to \F(E)$ be the unique 
extension of $\zeta$ into a derivation of
the free $\ttP$-operad~$\F(E)$. Then
\begin{equation}
\label{M1 se nabourala.}
\des^*(\underline\varpi) = \varpi.
\end{equation}
\end{subequations}
The verification uses the obvious fact that $\des^*(-)$ brings
derivations of non-$\Sigma$ operads into derivations of the
corresponding $\ttP$-operads, thus
both  $\des^*(\underline\varpi)$ and $\varpi$ are derivations of
$\ttP$-operads extending the same linear map $E \to \F(E)$.

Let us prove~(i). The underlying non-dg
$\ttP$-operads $\des^*(\uF(\desusp \oucC))$  resp.\ $\F(\desusp \ocC)$ 
of $\des^*(\ucobar(\ucC))$ resp.~$\cobar(\oC)$
are isomorphic by Proposition~\ref{Leopold}. It remains to verify that
$\des^*(\pa_{\ucobar}) =  \pa_\cobar$.
The differential $\pa_{\ucobar}$  in $\ucobar(\ucC)$ 
is, classically, the unique extension of the composition
\begin{equation}
\label{Mam porad rymu.}
\underline \zeta :=\
\desusp \oucC \stackrel \uparrow\longrightarrow \oucC
\stackrel{\overline{\underline\Delta}}\longrightarrow
\uF^2(\oucC) \stackrel\cong\longrightarrow \,\uparrow^2 \uF^2(\desusp \oucC)
\stackrel{\downarrow^2}\longrightarrow \uF^2(\desusp \oucC)
\end{equation}
whose constituents have, as we believe, obvious meanings, into a
degree $-1$ derivation of the free non-$\Sigma$ operad 
$\uF(\desusp \oucC)$. Likewise, we defined 
$\pa_\cobar$ as the extension of the composition $\zeta :\, \desusp \ocC  \to 
\F^2(\desusp \ocC)$ in~\eqref{Za_chvili_mi_zacina_seminar.} into a
derivation of  \hbox{$\F(\desusp \ocC)$}. It
follows from~\eqref{Mam prohlidku.} and the obvious equality
$\overline{\Delta}=  \des^*(\overline{\underline\Delta})$ that
$\des^*(-)$ brings the chain of maps in~\eqref{Mam porad rymu.}
into the chain in~\eqref{Za_chvili_mi_zacina_seminar.}, 
therefore $\des^*(\underline \eta) = \eta$, thus
$\des^*(\pa_{\ucobar}) =  \pa_\cobar$ by~\eqref{M1 se nabourala.}.

To prove (ii), recall that $\uD(\uoP) =
\ucobar(\uoP^*)$ by definition, therefore
\[
\des^* (\uD(\uoP)) = 
\des^*(\ucobar(\uoP^*)) \cong \cobar(\des^* (\uoP^*))
\]
by~(i), while 
\[
\cobar(\des^* (\uoP^*)) =
\cobar ((\des^*(\uoP))^*) = \cobar (\oP^*) =    \D(\oP)
\]
by the first equality 
of~\eqref{Musim koupit baterky.} and the definition of $ \D(\oP)$.

Let us finally attend to~(iii). 
If $\uoP$ is binary quadratic, then $\oP$ is binary quadratic  by
Proposition~\ref{Lepim_si_L-410.}. 
Recall \cite[Definition~4.1.3]{ginzburg-kapranov:DMJ94} that 
the non-$\Sigma$ operad $\uoP$ is Koszul if the
canonical~map\footnote{The authors of \cite{ginzburg-kapranov:DMJ94}
used slightly different conventions regarding the duals and suspensions,
but our version, accommodated to the needs of the present article, is
equivalent.} 
\[
\ucan:  \uD(\uoP^!) \longrightarrow (\uoP,0) 
\]
induces an isomorphism of homology, while the Koszulity of $\oP$
requires, by Definition~\ref{hodnoceni snad odlozeno}, 
the same for the canonical morphism 
\[
\can:  \D(\oP^!) \longrightarrow (\oP,0) .
\]
Under the identifications 
\[
\D(\oP^!) = \D(\des^*(\uoP)^!) = \D(\des^*(\uoP^!)) \cong   \des^*(\uD(\uoP^!))
\]
that follow from the second part of Proposition~\ref{Lepim_si_L-410.}
and the already established  item~(ii) of
Theorem~\ref{Je_teprve_utery_a_uz_padam_na_usta.},  
one clearly has that $\can  =\des^*(\ucan)$  which,
combined with the second equality of~\eqref{Musim koupit baterky.},
gives  \hbox{$H_*(\can) = \des^*(H_*(\ucan))$}. Thus the map $\can$ is a homology
isomorphism if 
\[
H_*(\can(\alpha)) =  H_*(\ucan(\des( \alpha)))
\]
is an isomorphism for each $\alpha \in \ttP$.
Since $\des : \ttP \to \Ord$ is an epimorphism on objects,
this happens if and only if $H_*(\ucan(\underline n))$ is an
isomorphism for each $\underline n \in \Ord$, i.e.\ when $\ucan$
induces an isomorphism of homology.
\end{proof}

\begin{corollary}
\label{Zavolam_Jarce?}
The terminal $\ttP$-operad $\Pterm$ is Koszul.
\end{corollary}

\begin{proof}
One way of proving the statement would be to identify $\D(\Pterm^!)$
with the minimal model $\min$ described in
Proposition~\ref{Vcera_jsem_se_vratil_z_Ciny.}. 
The corollary however follows  from
Theorem~\ref{Je_teprve_utery_a_uz_padam_na_usta.} since $\Pterm$ is
the restriction of the terminal non-$\Sigma$ operad $\Ass$ whose
Koszulity is superclassical.
\end{proof}


\def\cprime{$'$}\def\cprime{$'$}

\end{document}